\documentclass[11pt,a4paper]{article}
\usepackage{}
\usepackage{bbding}
\usepackage{mathrsfs}
\usepackage{amssymb}
\usepackage{amsfonts}
\usepackage{amsmath}
\usepackage{cases}
\usepackage{amsthm}
\usepackage[square, comma, sort&compress, numbers]{natbib}

\usepackage[top=1in,bottom=1in,left=1.20in,right=1.20in]{geometry}

\newtheorem{theorem}{Theorem}[section]
\newtheorem{lemma}{Lemma}[section]
\newtheorem{definition}{Definition}[section]
\newtheorem{proposition}[lemma]{Proposition}
\newtheorem{remark}{Remark}[section]

\numberwithin{equation}{section}
\theoremstyle{plain}

\newcommand{\me}{\mathrm{e}}
\newcommand{\mi}{\mathrm{i}}
\newcommand{\dif}{\mathrm{d}}

\DeclareMathOperator{\meas}{meas}
\DeclareMathOperator{\card}{card}

\begin{document}
\title{Quasi-periodic Solutions of a Derivative Nonlinear Schr\"{o}dinger Equation}
\author{Jie Liu\footnote{The work was
supported by the NNSF of China (Grant No. 11171185) and the NSF of Shandong Province (Grant No. ZR2010AM013).}\\
School of  Mathematics, Shandong University,
Jinan 250100, P.R. China.\\
E-mail: jzyzliujie@gmail.com.}
\date{}
\maketitle
\begin{abstract}
This paper is concerned with a one dimensional (1D) derivative nonlinear Schr\"{o}dinger equation
with periodic boundary conditions
\begin{equation*}
       \mi u_t+u_{xx}+\mi |u|^2u_x=0, \ \ x\in \mathbb{T}:=\mathbb{R}/2\pi\mathbb{Z}.
\end{equation*}
We show that above equation admits a family of real analytic quasi-periodic solutions with two Diophantine frequencies. The proof is based
on a partial Birkhoff normal form and KAM method.
\end{abstract}



\section[]{Introduction and Main Result}
In this paper, we consider the derivative nonlinear Schr\"{o}dinger equation (Chen-Lee-Liu-equation \cite{Chen})
\begin{equation}\label{DNLS4}
       \mi u_t+u_{xx}+\mi |u|^2u_x=0
\end{equation}
with periodic boundary condition
\begin{equation}\label{periodic  boundary condition}
      u(t,0)=u(t,2\pi),
\end{equation}
which appears in studies of ultrashort optical pulses. Moreover, Eq.\,\eqref{DNLS4}
has several applications in e.g. plasma physics and nonlinear fiber optics referring to \cite{Kodama} and \cite{Mjolhus}.

Consider the Hamiltonian partial differential equation
\begin{equation*}
    \dot{w}= Aw + F(w).
\end{equation*}
For some Sobolev space $\mathcal {H}^p \ni w$, linear operator $A$ maps $\mathcal {H}^p$ to $\mathcal {H}^{p-d}$ and
nonlinear term $F$ sends some neighborhood of $\mathcal {H}^p$ to $\mathcal {H}^{p-\delta}$.
One calls $d$ and $\delta$ the orders of $A$ and $F$ respectively.

When $\delta\leqslant 0$, the vector field $F$ is called bounded perturbation. The existence of quasi-periodic solutions of such PDEs has been widely investigated by many authors
\cite{Berti,Bourgain1,Bourgain2,Chierchia,Craig,Eliasson,Geng2,Geng1,Grebert,Kuksin1,kuksin4,kuksin5,Kuksin2,Poschel1,Poschel3,Si,Wayne,Yuan3,Yuan4,Zhang}.

When $\delta>0$, the vector field $F$ is called unbounded perturbation. Unlike the bounded case, there are few results of KAM theory for partial differential equations with unbounded perturbation.
The first KAM theorem for unbounded perturbations is due to Kuksin \cite{Kuksin6,Kuksin3} under the assumption $0< \delta < d-1$.
See also Kappeler and P\"{o}schel \cite{Kappeler}. Another KAM theorem with unbounded linear Hamiltonian perturbation is due to
Bambusi and Graffi \cite{Bambusi} which consider the time dependent linear Schr\"{o}dinger equation.

When $0< \delta = d-1$, which is called `` the limiting case", the nonlinearity of the PDE is the strongest. Recently, Liu and Yuan \cite{Yuan1} give a
theorem which generalizes Kuksin's theorem from $\delta< d-1$ to $\delta\leqslant d-1$.
In their paper, they still consider the homological equations of variable coefficients:
\begin{equation}\label{homological equation 0.1}
    -\mi\partial_{\omega}u+ \lambda u+\mu(\theta)u = p(\theta), \ \ \ |\text{Im}\,\theta|<s,
\end{equation}
Using the generalized Kuksin's theorem, Liu and Yuan \cite{Yuan2} establish an improved KAM theorem
which can prove the existence of quasi-periodic solution of a class of derivative nonlinear Schr\"{o}dinger  equations (DNLS)
\begin{equation}\label{DNLS1}
\mi u_t+u_{xx}-M_{\xi}u+\mi f(u,\bar{u})u_x=0,
\end{equation}
with Dirichlet boundary conditions,
where $f(u,\bar{u})$ be a analytic function in $\mathbb{C}^{2}$ with
\begin{equation*}
    \overline{ f(u,\overline{u})}=f(u,\overline{u}),  f(-u,-\overline{u})=-f(u,\overline{u}).
\end{equation*}

Then, Geng and Wu \cite{Geng} consider the derivative nonlinear Schr\"{o}dinger  equation
\begin{equation}\label{DNLS3}
       \mi u_t-u_{xx}-\mi (|u|^4u)_x=0
\end{equation}
with periodic boundary condition. Unlike \cite{Yuan2}, by using the compact form and the gauge invariant property, the homological equation \eqref{homological equation 0.1} becomes into
the following forms:
\begin{equation}\label{homological equation 0.2}
    -\mi\partial_{\omega}u+ \lambda u= p(\theta).
\end{equation}
Since normal form obtained in \cite{Geng} is independent
of the angle variables $\theta$, it is different from Kuksin's theorem \cite{Kuksin3} and Liu and Yuan's theorem \cite{Yuan2}. Then, using an abstract KAM theorem with
angle independent normal form, they obtain the real analytic quasi-periodic solutions for the derivative
nonlinear Schr\"{o}dinger equation \eqref{DNLS3} with only two Diophantine frequencies.

Lately, for a class of derivative nonlinear Schr\"{o}dinger  equation
\begin{equation}\label{DNLS2}
       \mi u_t+u_{xx}+\mi (f(|u|^2)u)_x=0,
\end{equation}
Liu and Yuan \cite{Yuan5} prove that Eq.\eqref{DNLS2} with periodic boundary conditions admits many $C^{\infty}$ (not real analytic) quasi-periodic solutions
with $N$ Diophantine frequencies, where $N$ is any positive integer. It is worth to note that the momentum conservation plays an important role in their results.
To use both Kuksin's lemma in \cite{Kuksin6} and the estimates in \cite{Yuan1}, the homological equations  must be  scalar i.e. the normal frequency
$\Omega_{j}$ is required to be simple $\Omega_{j}^{\sharp}=1$. So the KAM theorem for
unbounded perturbations in \cite{Yuan2} can not be used to the derivative nonlinear Schr\"{o}dinger
equation\eqref{DNLS2} with periodic boundary conditions, since the multiplicity $\Omega_{j}^{\sharp}=2$. But this difficulty can be avoided since
the nonlinear $\mi (f(|u|^2)u)_x$ does not contain the space variable $x$ explicitly, so that momentum is conserved for \eqref{DNLS2}. More details can be found in \cite{Yuan5}.

In this paper, we consider the derivative nonlinear Schr\"{o}dinger equation \eqref{DNLS4}
\begin{equation*}
       \mi u_t+u_{xx}+\mi |u|^2u_x=0
\end{equation*}
with periodic boundary condition.
Obviously, Eq.\,\eqref{DNLS4} is not contained in Eq.\,\eqref{DNLS1}, which is our first motivation to consider the quasi-periodic
solutions of \eqref{DNLS4}.

For \eqref{DNLS2}, if $f$ is the identity function, i.e. $f(z)=z$, then \eqref{DNLS2} reduces to
\begin{equation}\label{DNLS5}
       \mi u_t+u_{xx}+\mi (|u|^2u)_x=0,
\end{equation}
which appears in various physical applications and has been widely studied in the literature.
Applying the gauge transformation ((2.12) in \cite{Wadati})
\begin{equation*}
    v=u(x)\,\text{exp}\,\left\{-\frac{\mi}{2}\int^{x}_{-\infty}|u(\eta)|^2 \dif\, \eta\right\},
\end{equation*}
above Eq. \eqref{DNLS5} is transformed into Eq.\,\eqref{DNLS4}. However, in \cite{Xu}, they point that
\begin{quote}
    ``But the gauge transformation can't preserve the reduction conditions in spectral problem of the Kaup and Newell (KN)
    \cite{Kaup} system and involve complicated integrations. So it deserves to be investigated separately.''
\end{quote}
This is our second motivation to consider \eqref{DNLS4}.

To obtain the real analytic quasi-periodic solutions of Eq.\,\eqref{DNLS4}, we construct a KAM iteration for Hamiltonian PDEs with some special perturbations
which admits the compact form and the gauge invariant property like \cite{Geng}.

Assume that
\begin{equation}\label{zero value condition}
    [u]:=\frac{1}{2\pi}\int_{0}^{2\pi}u\,\dif\,x = 0,
\end{equation}
then the main result is described as follows:
\begin{theorem}\label{theorem 1}
Consider the derivative nonlinear Schr\"{o}dinger equation \eqref{DNLS4} with periodic boundary
conditions \eqref{periodic  boundary condition} and \eqref{zero value condition}. Fix $n_{1}$, $n_{2}$ satisfying that $n_{1}$ is odd and $|n_{2}-n_{1}|=4$. Then there exists a Cantor subset
$\mathcal {O}_{\ast} = \mathcal {O}_{\ast}(n_{1},n_{2}) \subset \mathbb{R}^{2}_{+}$ of positive Lebesgue measure, such that
each $\xi \in \mathcal {O}_{\ast}$ corresponds to a real analytic, quasi-periodic solution
\begin{equation*}
    u(t,x)=\sum\limits_{j=1}^{2}\sqrt{\frac{1}{2\pi}\xi_{j}}e^{\mi(\omega_{*j}t+n_{j}x)}+ O(|\xi|^{\frac{3}{2}})
\end{equation*}
of \eqref{DNLS4}, \eqref{periodic  boundary condition}, \eqref{zero value condition} with two Diophantine frequencies
\begin{equation*}
    \omega_{\ast j}= n_{j}^2+ O(|\xi|), \ \ 1 \leqslant j\leqslant 2.
\end{equation*}
Moreover, the quasi-periodic solutions u are linearly stable and depend on $\xi$ Whitney smoothly.
\end{theorem}

\begin{remark} Note that the solution which we obtain in Theorem \ref{theorem 1} is real analytic like \cite{Geng}, although the number of frequencies is only 2 not any positive
integer $N$. The reason is that in the KAM iteration, we still let $s_{m}$,
the radius of $Im\,\theta$, such that $s_{m}\rightarrow \frac{s}{2}$ as $m\rightarrow\infty$.
\end{remark}

The rest of the paper is organized as follows: In section 2, we give some definition such as
compact form and gauge invariant property. Although all the definitions are the same as \cite{Geng}, we would like to list
them here for reader's convenience.
In Section 3, we will give the Hamiltonian setting
corresponding to the Eq.\,\eqref{DNLS4} and derive a partial Birkhoff normal form of order
four for the lattice Hamiltonian. In Section 4, we will show some conditions about frequencies and
perturbation for the lattice Hamiltonian obtained in Section 3. In Section 5, we will give details for one step KAM iteration, summarize as an iteration lemma
and prove its convergence. At last, we give
the necessary measure estimate for the parameter set. Some technical lemmas necessary are given in the Appendix.

\section[]{Preliminary}
Denote $\mathbb{Z}_{\ast}=\mathbb{Z}\setminus \{0\}$,
for any integer $a > 0$ and $p\geqslant 0$, we introduce the phase space, complex valued functions space on $\mathbb{T}=\mathbb{R}/2\pi \mathbb{Z}$:
\begin{equation*}
    \mathcal {H}^{a,p}=\left\{u\in L^2(\mathbb{T},\mathbb{C}): \|u\|_{a,p}^2=\sum\limits_{n\in \mathbb{Z}_{\ast}}|\hat{u}_{n}|^2|n|^{2p}\me^{2a|n|}< +\infty\right\},
\end{equation*}
where $u=\sum_{n\in \mathbb{Z}_{\ast}}\hat{u}_{n}\me^{\mi nx}$ is the discrete Fourier transform.

Let $\ell^{a,p}$ be the space of all bi-complex valued sequences $q=(\cdots,q_{-2}, q_{-1}, q_1,q_2,\cdots)$ with
\begin{equation*}
    \|q\|^2_{a,p}=\sum\limits_{n\in \mathbb{Z}_{\ast}}|q_n|^2|n|^{2p}\me ^{2a|n|}<+ \infty.
\end{equation*}
The convolution $w*z$ of two such sequences is defined by $(w*z)_{n}=\sum_{m}w_{n-m}z_{m}$.

\begin{lemma}\emph{\cite{Kuksin1}}
For $a > 0, \ p > \frac{1}{2}$, the space $\ell^{a,p}$ is a Banach algebra with respect to convolution
of sequences, and
\begin{equation*}
    \|w*z\|_{a,p}\leq c\|w\|_{a,p}\|z\|_{a,p},
\end{equation*}
with a constant $c$ depending only on $p$.
\end{lemma}
In the following, we give the same definitions of compact form and gauge invariant property in \cite{Geng}. To keep the continuity and
enhance the readability, we list corresponding definitions and properties here.

Let
\begin{equation*}
    \mathcal {J}=\{\{n_1,n_2\}\in \mathbb{Z}_{\ast}| n_{1}\ \text{is odd}\ \ \text{and}\ |n_{2}-n_{1}|=4\}.
\end{equation*}
Without loss of generality, we assume that $n_{2} > n_{1}> 0 $ for simplicity.

\begin{definition}\emph{\cite{Geng}} Given $\{n_1,n_2\}\in \mathcal {J}$. A real analytic function
\begin{equation*}
    F=F(\theta, I, z, \bar{z})=\sum\limits_{k,\alpha,\beta}F_{k,\alpha,\beta}e^{\mi\langle k,\theta\rangle}z^{\alpha}\bar{z}^{\beta}
\end{equation*}
is said to admit a \textbf{compact form} with respect to $n_1, n_2$ if
\begin{equation*}
    F_{k,\alpha,\beta}=0, \ \ \text{whenever}\ k_1n_1+k_2n_2+\sum\limits_{n}(\alpha_n-\beta_n)n\neq 0,
\end{equation*}
where $k=(k_1,k_2)\in \mathbb{Z}^2$ and $\alpha=(\cdots,\alpha_n,\cdots), \beta=(\cdots,\beta_n,\cdots)$, $\alpha_n, \beta_n\in \mathbb{N}$, with finitely many
nonzero components of positive integers.
\end{definition}
Consider the Possion bracket
\begin{equation*}
    \{F,G\}=\sum\limits_{1\leqslant j\leqslant 2}\frac{\partial F}{\partial \theta_j}\frac{\partial G}{\partial I_j}-\frac{\partial F}{\partial I_j}\frac{\partial G}{\partial \theta_j}+
    \mi\sum\limits_{j\in \mathbb{Z}}\frac{\partial F}{\partial z_j}\frac{\partial G}{\partial \bar{z}_j}-\frac{\partial F}{\partial \bar{z}_j}\frac{\partial G}{\partial z_j},
\end{equation*}
we have the following lemma
\begin{lemma}\label{compact form}
\emph{\cite{Geng}} Given $\{n_1,n_2\}\in \mathcal {J}$ and consider two real analytic functions $F(\theta, I, z, \bar{z})$,
$G(\theta, I, z, \bar{z})$. If both $F$ and $G$ have compact forms with respect to $n_1, n_2$, then so does $\{F, G\}$.
\end{lemma}
\begin{definition}\emph{\cite{Geng}} A real analytic function
\begin{equation*}
    F=F(\theta, I, z, \bar{z})=\sum\limits_{k,\alpha,\beta}F_{k,\alpha,\beta}e^{\mi\langle k,\theta\rangle}z^{\alpha}\bar{z}^{\beta}
\end{equation*}
is said to admit \textbf{gauge invariant property} if
\begin{equation*}
    F_{k,\alpha,\beta}=0, \ \ \text{whenever}\ k_1+k_2+\sum\limits_{n}(\alpha_n-\beta_n)\neq 0,
\end{equation*}
where $k=(k_1,k_2)\in \mathbb{Z}^2$ and $\alpha=(\cdots,\alpha_n,\cdots), \beta=(\cdots,\beta_n,\cdots)$, $\alpha_n, \beta_n\in \mathbb{N}$, with finitely many
nonzero components of positive integers.
\end{definition}

\begin{lemma}\label{gauge invariant property}\emph{\cite{Geng}}
Consider two real analytic functions $F(\theta, I, z, \bar{z})$,
$G(\theta, I, z, \bar{z})$. If both $F$ and $G$ admit the gauge invariant property, then so does $\{F, G\}$.
\end{lemma}

\begin{lemma}\label{The special form}\emph{\cite{Geng}} Given $\{n_1,n_2\}\in \mathcal {J}$ and consider a real analytic functions $F(\theta, I, z, \bar{z})$. If $F$
has compact form and admits gauge invariant property with respect to $n_1, n_2$,
then $F$ contains no terms of the form $e^{\mi\langle k, \theta\rangle}z_n\bar{z}_n$ with $k\neq 0$ and $e^{\mi\langle k, \theta\rangle}z_n\bar{z}_m$ with $k=0$ and $n\neq m$.
\end{lemma}
Although this lemma has been proved in \cite{Geng}, in order to make the reader understand the role which compact form and gauge invariant property play in KAM
iteration, we would like to ``prove'' it again.
\begin{proof}
Consider $F(\theta, I, z, \bar{z})=\sum\limits_{k,\alpha,\beta}F_{k,\alpha,\beta}e^{\mi\langle k,\theta\rangle}z^{\alpha}\bar{z}^{\beta}$ with $\alpha=\beta=e_{n}$,
where $e_n$ denotes the $n$-th component being 1 and the other components being 0. Since $F$ has compact form with respect to $n_1, n_2$ and admits the gauge invariant property, we have
\begin{equation*}
\left\{\begin{split}
    &k_1n_1+k_2n_2+n-n&=0,\\
    &k_1+k_2+1-1&=0.
\end{split}
\right.
\end{equation*}
In View of $n_1\neq n_2$, we obtain that $k_1=k_2=0$.
Then consider $F(\theta, I, z, \bar{z})$ with $\alpha=e_{n}, \beta=e_{m}$ and $k=0$. Since $F$ has compact form with respect to $n_1, n_2$, we obtain
\begin{equation*}
    n-m=0.
\end{equation*}
Hence, lemma is proved.
\end{proof}
\begin{remark} From the proof of Lemma \ref{The special form}, we can find that the method in \cite{Geng} restrict the number of frequencies of quasi-periodic solution
to only 2, unlike in \cite{Yuan5} any positive number $N$.
\end{remark}

\begin{remark} We should show that compact form and the gauge invariant property will be preserved along KAM iterations. These properties enable simplify the homological equation in each KAM step.
\end{remark}
We denote
\begin{equation*}
    \mathcal {A}_{n_1,n_2}=\left\{P: P=\sum\limits_{k\in \mathbb{Z}^2, l\in\mathbb{N}^2, \alpha, \beta}P_{k,l,\alpha,\beta}e^{\mi\langle k,\theta\rangle}I^{l}z^{\alpha}\bar{z}^{\beta} \right\},
\end{equation*}
where $k, \alpha, \beta$ have the following relations:
\begin{equation*}
          k_1n_1+k_2n_2+\sum\limits_{n}(\alpha_n-\beta_n)n=0\ \text{and}\ k_1+k_2+\sum\limits_{n}(\alpha_n-\beta_n)=0.
\end{equation*}

\section[]{Hamiltonian and Normal Form}

In this section, We will study \eqref{DNLS4} as a Hamiltonian system on some suitable phase space $\mathscr{P}$.  Using the Hamiltonian formulation, we rewrite \eqref{DNLS4} with
periodic boundary condition in the Hamiltonian form
\begin{equation}\label{Hamiltion1}
    u_{t}=-\mi \frac{\partial H}{\partial \bar{u}},
\end{equation}
with Hamiltonian
\begin{equation}\label{Hamiltonian1}
    H=\int^{2\pi}_0 |u_x|^2\,\dif x - \frac{\mi}{2}\int^{2\pi}_0 |u|^2\bar{u}u_x\,\dif x,
\end{equation}
where the gradient is defined with respect to inner product in $L^2$: $\langle u, v\rangle= \int^{2\pi}_{0}u\bar{v}\, \dif x$.

Consider operator $A=-\partial_{xx}$ with the periodic boundary condition. The eigenfunctions is $\{\phi_{j}(x)=\sqrt{\frac{1}{2\pi}}\me^{\mi jx}\}$
and corresponding eigenvalue is $\lambda_{j}=j^2$.

To write it in infinitely many coordinates, we make the ansatz
\begin{equation}\label{ansatz}
    u=\mathscr{L}q=\sum\limits_{j\in \mathbb{Z}_{\ast}}q_j(t)\phi_j(x).
\end{equation}
The coordinates are taken from Hilbert space $\ell^{a,p}$. Due to the definition of spaces, there is an isomorphism
$\mathcal {L}: \ell^{a,p} \longmapsto \mathcal {H}^{a,p}$ with $\|u\|_{a,p}^2=\|\bar{u}\|_{a,p}^2=\|q\|_{a,p}^2$, for each $p \geqslant 0$.

Fixed $a > 0$ and $p>\frac{3}{2}$ in the following, one obtains the Hamiltonian
\begin{equation}\label{Hamiltonian2}
\begin{split}
    H=\Lambda+ G
\end{split}
\end{equation}
with
\begin{equation*}
    \Lambda=\sum\limits_{j\in \mathbb{Z}_{\ast}}\lambda_{j}|q_j|^2,
\end{equation*}
\begin{equation*}
    G=-\frac{\mi}{2}\int^{2\pi}_0 |\mathscr{L}q|^2(\overline{\mathscr{L}q})(\mathscr{L}q)_x\,\dif x,
\end{equation*}
on the phase space $\ell^{a,p}$ with symplectic structure $-\mi\sum_{j\in \mathbb{Z}_{\ast}}\dif q_{j}\wedge \dif \bar{q}_{j}$.
Its equation of motion are
\begin{equation}\label{Hamiltion2}
    \dot{q}_j=-\mi \frac{\partial H}{\partial \bar{q}_j},\ j\in \mathbb{Z}_{\ast}.
\end{equation}
They are the classical Hamiltonian equation of motion for the real and imaginary parts of $q_{j} = x_{j} + \mi y_{j}$ written
in complex notion.

\begin{lemma}Let $a > 0$ and $p \geqslant 0$. If a curve $\mathbb{R} \rightarrow \ell^{a,p}$, $t\rightarrow q(t)$ is a real analytic solution of
\eqref{Hamiltion2}, then
\begin{equation*}
    u=\mathscr{L}q=\sum\limits_{j\in \mathbb{Z}_{\ast}}q_j(t)\phi_j(x).
\end{equation*}
is a solution of \eqref{DNLS4} that is real analytic on $\mathbb{R}\times [0,2\pi]$.
\end{lemma}
\begin{proof} The proof is similar to Lemma 1 in \cite{Kuksin2}, we omit it.
\end{proof}

Then we establish the regularity of nonlinear Hamiltonian vector field $X_{G}$.
The perturbation term $G$ has the following properties:
\begin{lemma}\label{algebra lemma}
For $a > 0$ and $p>\frac{3}{2}$, the function $G$ is analytic in some neighborhood of the origin in $\ell^{a,p}$ with real value,
and $G_{\bar{q}}$ is an analytic map from some neighborhood of the origin in $\ell^{a,p}$ into $\ell^{a,p-1}$
with
\begin{equation}
    \|G_{\bar{q}}\|_{a,p-1}=O(\|q\|^3_{a,p}).
\end{equation}
\end{lemma}
\begin{proof} Let $G_{\bar{q}}=(\{\frac{\partial G}{\partial \bar{q}_{l}}\})$, where
\begin{equation*}
\begin{split}
    \frac{\partial G}{\partial \bar{q}_{l}}&=-\mi\int^{2\pi}_0 |u|^2u_x\bar{\phi}_{l}\,\dif x, \ \ u=\mathscr{L}q.
\end{split}
\end{equation*}
Let $q$ be in $\ell^{a,p}$, then $(jq_{j})_{j\in \mathbb{Z}_{\ast}}\in \ell^{a,p-1}$.
By the algebra property, we can get
\begin{equation*}
    \||u|^2u_{x}\|_{a,p-1}\leqslant c \|u\|^3_{a,p}.
\end{equation*}
The components of the gradient $G_{\bar{q}}$ are its Fourier coefficients, so $G_{\bar{q}}$ in $\ell^{a,p-1}$, with
\begin{equation*}
    \|G_{\bar{q}}\|_{a,p-1}\leqslant \||u|^2u_{x}\|_{a,p-1}\leqslant c \|u\|^3_{a,p} \leqslant c \|q\|^3_{a,p}.
\end{equation*}
The regularity of $G_{\bar{q}}$ follows from the regularity of its components.

\end{proof}

For the nonlinearity $\mi|u|^2u_{x}$, we find
\begin{equation*}
    G=\frac{1}{2}\sum\limits_{i,j,k,l}jG_{ijkl}q_{i}q_{j}\bar{q}_{k}\bar{q}_{l}=\sum\limits_{\alpha,\beta}G_{\alpha,\beta}q^{\alpha}\bar{q}^{\beta},
\end{equation*}
where
\begin{equation*}
    G_{ijkl}=\int^{2\pi}_{0}\phi_{i}\phi_{j}\bar{\phi}_{k}\bar{\phi}_{l}\ \dif x=
    \left\{\begin{aligned}
             \frac{1}{2\pi}, &\ \ \ \text{if}\ i+j=k+l,\\
             0,  &\ \ \ \text{otherwise}.
    \end{aligned} \right.
\end{equation*}
\begin{remark}
From above special forms of $G$ and $G_{ijkl}$, we know that $G\in \mathcal {A}_{n_1,n_2}$, i.e. $G_{\alpha, \beta}\neq 0$ when $\sum_{n}(\alpha_n-\beta_n)n=0$ and
$\sum_{n}(\alpha_n-\beta_n)=0$.
\end{remark}
\begin{lemma}
If $i+j=k+l$ and $\{i,j\}\neq\{k,l\}$, then
\begin{equation*}
    \lambda_{i}+\lambda_{j}-\lambda_{k}-\lambda_{l}=i^2+j^2-k^2-l^2\neq 0.
\end{equation*}
\end{lemma}
\begin{proof}
Suppose $i^2+j^2=k^2+l^2$, we can get $ij=kl$. Since there are two real roots for quadratic polynomial at most, we can get
$\{i,j\}=\{k,l\}$. This is a contradiction.
\end{proof}

For all indices $i, j, k, l$ satisfying $i+j=k+l$, we denote
\begin{equation*}
\begin{split}
    \mathcal {N}&=\{(i, j, k, l)\in \mathbb{Z}_{\ast}^4|\{i,j\}=\{k,l\}\},\\
    \Delta_{l}&=\{(i, j, k, l)\in \mathbb{Z}_{\ast}^4|\ \text{there are right $l$ components not in}\{n_{1}, n_{2}\}\},\\
\end{split}
\end{equation*}
for $l=0,1,2$ and
\begin{equation*}
    \Delta_{3}=\{(i, j, k, l)\in \mathbb{Z}_{\ast}^4|\ \text{there are at least 3 components not in}\{n_{1}, n_{2}\}\}.
\end{equation*}

\begin{lemma}\label{small division condition1}
For fixed $n_{1}, n_{2}$, denote $N= \max \{|n_{1}|, |n_{2}|\}$.
Let $(i, j, k, l)\in (\Delta_{0}\setminus \mathcal {N})\cup\Delta_{1}\cup (\Delta_{2}\setminus \mathcal {N}):=\Delta$, i.e
there are at least $2$ components in $\{n_{1}, n_{2}\}$, if
\begin{equation*}
    i+j-k-l=0,
\end{equation*}
then
\begin{equation*}
    |\lambda_{i}+\lambda_{j}-\lambda_{k}-\lambda_{l}|=|i^2+j^2-k^2-l^2|\geqslant \frac{|j|}{N}.
\end{equation*}
\end{lemma}
\begin{proof}
As $i+j-k-l=0$, then by direct calculation, we obtain
\begin{equation*}
    i^2+j^2-k^2-l^2=2(j-k)(j-l).
\end{equation*}
Observe that $j\neq k, l$. Hence, if $|j|\leqslant 2N$, then
\begin{equation*}
    |2(j-k)(j-l)|\geqslant 2 \geqslant \frac{|j|}{N};
\end{equation*}
if $|j|> 2N$, at least one of $k, l$ being in $\{n_{1}, n_{2}\}$, then
\begin{equation*}
    |2(j-k)(j-l)|\geqslant 2(|j|-N) \geqslant \frac{|j|}{N}.
\end{equation*}
\end{proof}

\begin{lemma}\label{Birkhoff normal form theorem 1}
Given $\{n_1,n_2\}\in \mathcal {J}$, there exists a real analytic, symplectic change of coordinates $\Gamma$ in a
neighbourhood of the origin in $\ell^{a,p}$ which transforms hamiltonian $H= \Lambda + G $ into Birkhoff normal form up to order four. That is
\begin{equation*}
    H \circ \Gamma= \Lambda + \bar{G} + \hat{G} + K,
\end{equation*}
where $X_{\bar{G}}$, $X_{\hat{G}}$ and $X_{K}$ are real analytic vector fields from a neighbourhood of origin in $\ell^{a,p}$ to $\ell^{a,p-1}$,
\begin{equation*}
    \bar{G}=\frac{1}{4\pi}\sum\limits_{i,j\in \mathbb{Z}_{\ast}}j|q_{i}|^2|q_{j}|^2,
\end{equation*}
and
\begin{equation*}
     \|\hat{G}\|_{a, p-1} =O(\|q\|^{4}_{a,p}),\ \ \ \|K\|_{a, p-1} =O(\|q\|^{6}_{a,p}).
\end{equation*}
Moreover, $K(q,\bar{q})\in \mathcal {A}_{n_1,n_2}$.
\end{lemma}
\begin{proof}
Define
\begin{equation*}
    F=\frac{1}{2}\sum\limits_{i+j-k-l=0}F_{ijkl}q_{i}q_{j}\bar{q}_{k}\bar{q}_{l}
\end{equation*}
with coefficients
\begin{equation*}
     \mi F_{ijkl}=
     \left\{\begin{aligned}
             \frac{-jG_{ijkl}}{\lambda_{i}+\lambda_{j}-\lambda_{k}-\lambda_{l}}, &\ \ \ & (i,j,k,l)\in
             \Delta,\\
             0,  &\ \ \ &\text{otherwise}.
    \end{aligned} \right.
\end{equation*}
Then we have
\begin{equation*}
\begin{split}
    \{\Lambda, F\} + G &= \frac{1}{2}\sum\limits_{i+j-k-l=0}(jG_{ijkl}+\mi (\lambda_{i}+\lambda_{j}-\lambda_{k}-\lambda_{l})F_{ijkl})q_{i}q_{j}\bar{q}_{k}\bar{q}_{l}\\
                       &= \frac{1}{2}\sum\limits_{\substack{i+j-k-l=0\\(i,j,k,l)\in (\Delta_{0}\cap\mathcal {N})\cup(\Delta_{2}\cap\mathcal {N})}}jG_{ijkl}q_{i}q_{j}\bar{q}_{k}\bar{q}_{l} +
                       \frac{1}{2}\sum\limits_{\substack{i+j-k-l=0\\(i,j,k,l)\in \Delta_{3}}}jG_{ijkl}q_{i}q_{j}\bar{q}_{k}\bar{q}_{l}\\
                       &= \frac{1}{4\pi}\sum\limits_{i,j\in \mathbb{Z}_{\ast}}j|q_{i}|^2|q_{j}|^2 + \hat{G}\\
                       &= \bar{G} + \hat{G},
\end{split}
\end{equation*}
where $\{\cdot, \cdot\}$ is a Poisson bracket with respect to the symplectic structure $-\mi\sum_{j\in \mathbb{Z}_{\ast}}\dif q_{j}\wedge \dif \bar{q}_{j}$.
Letting $\Gamma= X_{F}^{1}$, then
\begin{equation*}
    \begin{split}
    H\circ \Gamma=&\left.H \circ X_{F}^{t}\right|_{t=1}\\
                 =&H+\{H,F\}+\int^1_0 (1-t)\{\{H,F\},F\}\circ X_{F}^{t}\,dt\\
                 =&\Lambda+\{\Lambda,F\}+G +\{G, F\}+\int^1_0 (1-t)\{\{H,F\},F\}\circ X_{F}^{t}\,dt\\
                 =&\Lambda+\bar{G}+ \hat{G} + K,\\
    \end{split}
\end{equation*}
where
\begin{equation*}
\begin{split}
    K=&\{G, F\}+\frac{1}{2!}\{\{\Lambda, F\}, F\}+\frac{1}{2!}\{\{G, F\}, F\}\\
      &+ \cdots + \frac{1}{n!}\{\cdots\{\Lambda, \underbrace{F\}\cdots, F}\limits_{n}\} + \frac{1}{n!}\{\cdots\{G, \underbrace{F\}\cdots, F}\limits_{n}\}+ \cdots.
\end{split}
\end{equation*}

Now we prove the analyticity of the preceding transformation $\Gamma$. First, note that when $(i, j, k, l) \in \Delta$, we have $|\lambda_{i}+\lambda_{j}-\lambda_{k}-\lambda_{l}|\geqslant
\frac{|j|}{N}$. So we know
\begin{equation*}
    \left|\frac{\partial F}{\partial \bar{q}_{l}}\right| \leqslant \sum\limits_{i+j-k=l}
    \left|\frac{jG_{ijkl}}{\lambda_{i}+\lambda_{j}-\lambda_{k}-\lambda_{l}}\right||q_iq_j\bar{q}_k|
    \leqslant c\sum\limits_{i+j-k=l}|q_iq_j\bar{q}_k|
    =c(q\ast q \ast \bar{q})_{l}.
\end{equation*}
Hence, by Lemma \ref{algebra lemma},
\begin{equation*}
    \|F_{\bar{q}}\|_{a,p} \leqslant c\|q* q * \bar{q}\|_{a,p} \leqslant c\|q\|_{a,p}^3.
\end{equation*}
The analyticity of $F_{\bar{q}}$ then follows from that of each of its component and its local boundedness.
Moreover, it is clear that $\|K\|_{a,p-1}\leqslant c\|q\|^6_{a,q}$. The analogue claims for $X_{\bar{G}}$ and $X_{\hat{G}}$ are obvious.

We note that $G$ and $F$ have  compact forms. Hence, by Lemma \ref{compact form}, $\{G,F\}$ has a compact form. Since $\Lambda$ is already in a compact form, repeating
applications of Lemma \ref{compact form} show that all terms of $K$ have compact forms, so does $K$. Similarly, using Lemma \ref{gauge invariant property}, we can get that $K$ has the gauge invariant
property. Hence
$K\in \mathcal {A}_{n_1,n_2}$.
\end{proof}
Now our Hamiltonian is $ \widetilde{H} = \Lambda + \bar{G} + \hat{G} + K$. Introduce the symplectic polar and complex coordinates by setting
\begin{equation*}
     \left\{\begin{aligned}
             &q_{n_j}=\sqrt{I_{j}+\xi_{j}}\me ^{\mi \theta_{j}}, &\ \ \ &j=1,2;\\
             &q_{j}=z_{j},  &\ \ \ &j\in \mathbb{Z}_{1}=\mathbb{Z}_{\ast}\setminus \{n_1,n_2\},
    \end{aligned} \right.
\end{equation*}
where $n_1\neq n_2$, $\xi=\{\xi_1,\xi_2\}\in \mathbb{R}^2_{+}$. Then
\begin{equation*}
    \Lambda=\sum\limits_{1\leqslant j\leqslant 2}\lambda_{n_{j}}(I_{j}+\xi_{j})+ \sum\limits_{j\in \mathbb{Z}_{1}}\lambda_{j}z_{j}\bar{z}_{j},
\end{equation*}
\begin{equation*}
\begin{split}
    4\pi\bar{G}=&\sum\limits_{1\leqslant i,j\leqslant 2}n_{j}(I_{i}+\xi_{i})(I_{j}+\xi_{j})+ \sum\limits_{1\leqslant i\leqslant 2, j\in \mathbb{Z}_{1}}j(I_{i}+\xi_{i})z_{j}\bar{z}_{j}\\
                &+\sum\limits_{1\leqslant j\leqslant 2, i\in \mathbb{Z}_{1}}n_j(I_{j}+\xi_{j})z_{i}\bar{z}_{i}
                +\sum\limits_{i,j \in \mathbb{Z}_{1}}jz_{i}\bar{z}_{i}z_{j}\bar{z}_{j},
\end{split}
\end{equation*}
where $(\theta, I)\in \mathbb{T}^2\times \mathbb{R}^2$ be standard angle-action variables in the $(q_{n_{1}}, q_{n_{2}}, \bar{q}_{n_{1}}, \bar{q}_{n_{2}})$ space around $\xi$. Then we get
\begin{equation*}
    -\mi\sum_{j\in \mathbb{Z}_{\ast}}\dif q_{j}\wedge \dif \bar{q}_{j}=\sum_{1\leqslant j\leqslant 2}\dif \theta_{j}\wedge \dif I_{j}-\mi\sum_{j\in \mathbb{Z}_{1}}\dif z_{j}\wedge \dif \bar{z}_{j},
\end{equation*}
and Possion bracket
\begin{equation*}
    \{F,G\}=\sum\limits_{1\leqslant j\leqslant 2}\frac{\partial F}{\partial \theta_j}\frac{\partial G}{\partial I_j}-\frac{\partial F}{\partial I_j}\frac{\partial G}{\partial \theta_j}
    -\mi\sum\limits_{j\in \mathbb{Z}_{1}}\frac{\partial F}{\partial z_j}\frac{\partial G}{\partial \bar{z}_j}-\frac{\partial F}{\partial \bar{z}_j}\frac{\partial G}{\partial z_j}.
\end{equation*}
The new Hamiltonian, still denoted by $\widetilde{H}$, up to a constant depending only on $\xi$, is given by
\begin{equation*}
    \widetilde{H}= N+ P =\langle \tilde{\omega}(\xi), I \rangle+ \sum\limits_{j\in \mathbb{Z}_{1}}\tilde{\Omega}_{j}(\xi)z_{j}\bar{z}_{j} + \widetilde{P}(I,\theta, z, \bar{z},\xi),
\end{equation*}
where $\tilde{\omega}(\xi)=(\tilde{\omega}_{1}(\xi), \tilde{\omega}_{2}(\xi))$ with
\begin{equation*}
    \tilde{\omega}_{1}(\xi)=\lambda_{n_{1}}+ \frac{1}{4\pi}((n_{1}+n_{1})\xi_{1}+ (n_{1}+n_{2})\xi_{2}),
\end{equation*}
\begin{equation*}
    \tilde{\omega}_{2}(\xi)=\lambda_{n_{2}}+ \frac{1}{4\pi}((n_{2}+n_{1})\xi_{1}+ (n_{2}+n_{2})\xi_{2}),
\end{equation*}
\begin{equation*}
    \tilde{\Omega}_{j}(\xi)=\lambda_{j}+ \frac{1}{4\pi}((n_{1}+j)\xi_{1}+ (n_{2}+j)\xi_{2}),\ \ \ j\in \mathbb{Z}_{1}.
\end{equation*}
At the same time, the perturbation is
\begin{equation}\label{perturbation}
    \begin{split}
    \widetilde{P}= &K + O(|I|^2) + O(|I|^2|\xi|)+O(|I|\sum\limits_{j\in \mathbb{Z}_{1}}|j||z_j|^2)+ O(|I|\sum\limits_{j\in \mathbb{Z}_{1}}|z_j|^2)
    + O(|I||\xi|\sum\limits_{j\in \mathbb{Z}_{1}}|z_j|^2)\\
    &+ O(\sum\limits_{i,j\in \mathbb{Z}_{1}}|z_i|^2|j||z_j|^2)
    + O(|\xi|^{\frac{1}{2}}\sum\limits_{i=1}^{2}\sum\limits_{i+j-k=n_{i}}|z_{i}||jz_{j}||\bar{z}_{k}|)
    + O(\sum\limits_{i+j-k-l=0}|z_{i}||jz_{j}||\bar{z}_{k}||\bar{z}_{l}|).
    \end{split}
\end{equation}

\begin{lemma}If $F(q,\bar{q})\in \mathcal {A}_{n_1,n_2}$, then after above symplectic polar and complex coordinates transform, $F$
is still in $\mathcal {A}_{n_1,n_2}$.
\end{lemma}
\begin{proof}
Suppose $F_{k, \alpha, \beta} \neq 0$, when $\sum_{n}(\alpha_n-\beta_n)n=0$ and $\sum_{n}(\alpha_n-\beta_n)=0$. Without of the loss of
generality, we consider term
\begin{equation*}
   F_{n_{1},j,k.l}q_{n_{1}}q_{j}\bar{q}_{k}\bar{q}_{l}, \ \ \ \text{with}\ \  n_{1}+j-k-l=0, \ \ j, k, l \in \mathbb{Z}_{1}.
\end{equation*}
By the symplectic polar and complex coordinates
transform, it becomes
\begin{equation*}
   F_{n_{1},j,k.l}\sqrt{I_{1}+\xi_{1}}e^{i\theta_{1}}z_{j}\bar{z}_{k}\bar{z}_{l}.
\end{equation*}
Obviously, it satisfies
\begin{equation*}
   \begin{split}
          &k_1n_1+k_2n_2+\sum\limits_{n}(\alpha_n-\beta_n)n=n_{1}+j-k-l=0,\\
          &k_1+k_2+\sum\limits_{n}(\alpha_n-\beta_n)=1+1-1-1=0.
   \end{split}
\end{equation*}
The argument of the others terms is analogous to above and we omit it, then we complete the lemma.
\end{proof}
Now, let $\varepsilon > 0$ be sufficiently small. Rescaling $\xi_{j}$ by $\varepsilon^4 \xi_{j}, j=1, 2$, $z, \bar{z}$ by $\varepsilon^3 z, \varepsilon^3 \bar{z}$,
and $I$ by $\varepsilon^6 I$, one obtains the rescaled Hamiltonian
\begin{equation}\label{Hamiltonian3}
\begin{split}
    H(I,\theta, z, \bar{z},\xi)&=\varepsilon^{-10}\widetilde{H}(\varepsilon^6I,\theta, \varepsilon^3 z, \varepsilon^3 \bar{z},\varepsilon^4 \xi)\\
    &=\langle \omega^{\ast}(\xi), I \rangle+ \sum\limits_{j\in \mathbb{Z}_{1}}\Omega_{j}^{\ast}(\xi)z_{j}\bar{z}_{j} +  \varepsilon P^{\ast}(I,\theta, z, \bar{z},\xi),
\end{split}
\end{equation}
where $\omega^{\ast}(\xi)=(\omega_{1}^{\ast}(\xi),\omega_{2}^{\ast}(\xi))$ with
\begin{equation}\label{normal frequence1}
    \omega_{1}^{\ast}(\xi)=\varepsilon^{-4}\lambda_{n_{1}}+ \frac{1}{4\pi}((n_{1}+n_{1})\xi_{1}+ (n_{1}+n_{2})\xi_{2}),
\end{equation}
\begin{equation}\label{normal frequence2}
    \omega_{2}^{\ast}(\xi)=\varepsilon^{-4}\lambda_{n_{2}}+ \frac{1}{4\pi}((n_{2}+n_{1})\xi_{1}+ (n_{2}+n_{2})\xi_{2}),
\end{equation}
\begin{equation}\label{tangent frequence}
    \Omega_{j}^{\ast}(\xi)=\varepsilon^{-4}\lambda_{j}+ \frac{1}{2\pi}((n_{1}+j)\xi_{1}+ (n_{2}+j)\xi_{2}),\ \ j\in \mathbb{Z}_{1}.
\end{equation}
The perturbation is
\begin{equation}\label{perturbation1}
    P^{\ast}(I,\theta, z, \bar{z},\xi)=\varepsilon^{-11}\widetilde{P}(\varepsilon^6 I,\theta, \varepsilon^3 z, \varepsilon^3 \bar{z},\varepsilon^4 \xi).
\end{equation}

Note that Eq. \eqref{DNLS4} has a conservation $\int_{0}^{2\pi} |u|^2\,\dif\,x= \sum_{j\neq 0} |q_j|^2=c,$
i.e.
\begin{equation*}
    |q_{n_1}|^2+|q_{n_2}|^2+\sum\limits_{j\in \mathbb{Z}_1} |q_j|^2=c.
\end{equation*}
The above rescaling yields that
\begin{equation*}
    \varepsilon^2 (I_1+ I_2)+ (\xi_1+ \xi_2)+\varepsilon^2 \sum\limits_{j\in \mathbb{Z}_1} |z_j|^2=c,
\end{equation*}
that is
\begin{equation*}
    \xi_1+ \xi_2=c-\varepsilon^2(I_1+ I_2+\sum\limits_{j\in \mathbb{Z}_1} |z_j|^2)=c+ O(\varepsilon^2).
\end{equation*}
Let $\omega(\xi)=(\omega_{1}(\xi),\omega_{2}(\xi))$, $\Omega= (\Omega_{j})_{j\in\mathbb{Z}_{1}}$, where
\begin{equation}
    \omega_{1}(\xi)=\varepsilon^{-4}n_{1}^2+ \frac{1}{4\pi}(n_1+n_2)c+\frac{1}{4\pi}(n_1-n_2)\xi_{1},
\end{equation}
\begin{equation}
    \omega_{2}(\xi)=\varepsilon^{-4}n_{2}^2+ \frac{1}{4\pi}(n_1+n_2)c+\frac{1}{4\pi}(n_2-n_1)\xi_{2},
\end{equation}
\begin{equation}
    \Omega_{j}(\xi)=\varepsilon^{-4}j^2+ \frac{1}{4\pi}(cj + n_{1}\xi_{1}+ n_{2}\xi_{2}).
\end{equation}
We can write \eqref{Hamiltonian3} as
\begin{equation}\label{Hamiltonian4}
    H(I,\theta, z, \bar{z},\xi)
    =\langle \omega(\xi), I \rangle+ \sum\limits_{j\in \mathbb{Z}_{1}}\Omega_{j}(\xi)z_{j}\bar{z}_{j} + P(I,\theta, z, \bar{z},\xi),
\end{equation}
where
\begin{equation}\label{perturbation 1}
    P=\varepsilon P^{*}-\varepsilon^2\frac{1}{2\pi}(n_1+n_2)(I_1+ I_2+\sum\limits_{j\in \mathbb{Z}_1} |z_j|^2)^2.
\end{equation}

\section[]{Some Conditions}
Consider the phase space
\begin{equation*}
    \mathscr{P}^{a,p}=\mathbb{T}^2 \times \mathbb{R}^2 \times \ell^{a,p} \times \ell^{a,p}
\end{equation*}
with the coordinates $(\theta, I, z, \bar{z})$. We denote $\mathscr{P}^{a,p}_{\mathbb{C}}$ and $\ell^{a,p}_{\mathbb{C}}$ the complexification of phase space $\mathscr{P}^{a,p}$
and $\ell^{a,p}$ respectively.
Define a neighborhood of $\mathbb{T}^2_0= \mathbb{T}^2 \times \{I=0\} \times \{z=0\} \times \{\bar{z}=0\}$ by
\begin{equation*}
\begin{split}
    D(s,r)&=\left\{(\theta,I,z,\bar{z}): |\text {Im}\ \theta|<s, |I|<r^2, \|z\|_{a,p}<r, \|\bar{z}\|_{a,p}<r \right\}\\
          &\subset \mathbb{C}^2 \times \mathbb{C}^2 \times \ell^{a,p}_{\mathbb{C}} \times \ell^{a,p}_{\mathbb{C}}=\mathscr{P}^{a,p}_{\mathbb{C}},
\end{split}
\end{equation*}
where $|\cdot|$ denotes the sup--norm for complex vectors,

Let $\mathcal {O}$ be a neighborhood of the origin in $\mathbb{R}^2_{+}$. Define the difference operator $\Delta_{\xi\zeta}$ in the variable
$\xi,\zeta \in \mathcal {O}$
\begin{equation*}
    \Delta_{\xi\zeta}=f(\cdot,\xi)-f(\cdot,\zeta).
\end{equation*}
We define the distance
\begin{equation*}
    |\Omega-\Omega'|_{-\delta, \mathcal {O}}=\sup\limits_{\xi \in \mathcal {O}}\sup\limits_{j\in \mathbb{Z}_1}j^{-\delta}|\Omega_{j}(\xi)-\Omega'_{j}(\xi)|,
\end{equation*}
and the Lipschitz semi-norm of frequencies $\omega$ and $\Omega$:
\begin{equation*}
    |\omega|^{\text{lip}}_{\mathcal {O}}=\sup\limits_{\substack{\xi,\zeta\in \mathcal {O},\\ \xi\neq \zeta}}\frac{|\Delta _{\xi,\zeta}\omega|}{|\xi-\zeta|},\ \ \ \
    |\Omega|^{\text{lip}}_{-\delta,\mathcal {O}}=\sup\limits_{\substack{\xi,\zeta\in \mathcal {O},\\ \xi\neq \zeta}}\sup\limits_{j\in \mathbb{Z}_1}\frac{j^{-\delta}|\Delta
    _{\xi,\zeta}\Omega_{j}|}{|\xi-\zeta|},
\end{equation*}
for any real number $\delta$. Denote $M=|\omega|^{\text{lip}}_{\mathcal {O}}+ |\Omega|^{\text{lip}}_{-\delta,\mathcal {O}}$.

For $l=(l_1,\ldots,l_n)\in \mathbb{Z}^{n}$, we define
\begin{equation*}
    |l|=\sum\limits_{j=1}^{n}|l_j|,\ \ |l|_{\delta}=\sum\limits_{j\neq 0}|j|^{\delta}|l_j|, \ \ \langle l\rangle_{\delta}=\max\{1, |\sum\limits_{j\neq 0}|j|l_j|\cdot |\sum\limits_{j\neq
    0}|j|^{\delta}|l_j|\},
\end{equation*}
and set
\begin{equation*}
    \Pi=\{(k, l) \neq 0, |l|\leqslant 2\}\subset \mathbb{Z}^2\times \mathbb{Z}^{\infty}.
\end{equation*}

We now prove some propositions of the Hamiltonian in the normal form \eqref{Hamiltonian3}.

\begin{proposition}\label{Nondegeneracy}The map $\xi \longmapsto \omega(\xi)$ is a homeomorphism from $\mathcal {O}$ to its image,
which is Lipschitz continuous and its inverse also. The functions
\begin{equation*}
    \xi \longrightarrow \frac{\Omega_{j}(\xi)}{|j|}
\end{equation*}
are uniformly Lipschtiz on $\mathcal {O}$ for $j\neq 0$. Morover, there exists a constant $m>0 $ such that for all $\xi \in \mathcal {O}$,
\begin{equation}\label{nonresonant condition}
    |\langle l,\Omega(\xi)\rangle|\geqslant m \langle l\rangle_{1}, \ \ \forall 1\leqslant |l|\leqslant 2.
\end{equation}
\end{proposition}
\begin{proof}
Rewrite $\omega(\xi), \Omega(\xi)$ as $\omega(\xi)= \alpha +A\xi$, $\Omega(\xi)= \beta + B\xi$, where
\begin{equation*}
    \alpha= \left(
              \begin{array}{c}
                 \varepsilon^{-4}n_1^2 + \frac{1}{4\pi}(n_1+n_2)c \\
                 \varepsilon^{-4}n_2^2 + \frac{1}{4\pi}(n_1+n_2)c\\
              \end{array}
            \right), \ \
    A=\frac{1}{4\pi}\left(
        \begin{array}{cc}
          n_1-n_2 & 0 \\
          0 & n_2-n_1 \\
        \end{array}
      \right),
\end{equation*}
and
\begin{equation*}
    \beta= \left(
             \begin{array}{c}
               \vdots \\
               \varepsilon^{-4}j^2+ \frac{1}{4\pi} cj \\
               \vdots \\
             \end{array}
           \right)_{j\in \mathbb{Z}_1}, \ \ \
    B=\frac{1}{4\pi}\left(
        \begin{array}{cc}
          \vdots & \vdots\\
          n_1  & n_2\\
          \vdots & \vdots\\
        \end{array}
      \right).
\end{equation*}
Since $det A= -\frac{1}{4\pi}(n_1-n_2)^2 \neq 0$, we have that $\langle k, \omega(\xi)\rangle \not\equiv 0$ for $k \neq 0$ and the map $\xi \longmapsto \omega(\xi)$ is a Lipschitz homeomorphism.

For $j\in \mathbb{Z}_{1}$
\begin{equation*}
       \left|\frac{\Omega_{j}(\xi)}{|j|}-\frac{\Omega_{j}(\zeta)}{|j|}\right|=\frac{1}{4\pi}\left|\frac{n_{1}(\xi_{1}-\zeta_{1})}{|j|}+
       \frac{n_{2}(\xi_{2}-\zeta_{2})}{|j|}\right|
       \leqslant \frac{\max\{|n_1|, |n_{2}|\}}{4\pi}|\xi-\zeta|.
\end{equation*}

At last, we prove inequality \eqref{nonresonant condition} in three cases:

Case I:  $|l|=1$, suppose $l_{j}=1, l_{i}=0$, for $i\neq j$, then
\begin{equation*}
\begin{split}
    |\langle l,\Omega(\xi)\rangle|=|\Omega_{j}(\xi)|&=|\varepsilon^{-4}j^2+ \frac{1}{4\pi}(cj + n_{1}\xi_{1}+ n_{2}\xi_{2})|\\
    &\geqslant |\varepsilon^{-4}- \frac{1}{4\pi}c||j|^2.
\end{split}
\end{equation*}

Case II:  $|l|=2$ and $l_{i}=1,l_{j}=1$, for $i\neq j$, then
\begin{equation*}
\begin{split}
    |\langle l,\Omega(\xi)\rangle|&=|\Omega_{i}(\xi)+\Omega_{j}(\xi)|\geqslant (\varepsilon^{-4}- \frac{1}{4\pi}c)(|i|^2+|j|^2)\\
    &\geqslant \frac{1}{4}|\varepsilon^{-4}- \frac{1}{4\pi}c|(|i|+|j|)^2.
\end{split}
\end{equation*}

Case III:  $|l|=2$ and $l_{i}=1,l_{j}=-1$, for $i\neq j$, then
\begin{equation*}
\begin{split}
    |\langle l,\Omega(\xi)\rangle|&=|\Omega_{i}(\xi)-\Omega_{j}(\xi)|=|\varepsilon^{-4}(i^2-j^2)+\frac{1}{4\pi}c(i-j)| \\
    &\geqslant |\varepsilon^{-4}- \frac{1}{4\pi}c|||i|^2-|j|^2|.
\end{split}
\end{equation*}
for properly selected $m>0$ we can show
\begin{equation*}
    |\langle l,\Omega(\xi)\rangle|\geqslant m \langle l\rangle_{1}.
\end{equation*}
Thus we prove this proposition.
\end{proof}

Setting $\mathcal {O}_{0}$ consisting of all $\xi \in \mathcal {O}$ such that
\begin{equation*}
    \begin{aligned}
      &|\langle k, \omega(\xi)\rangle|\geqslant \frac{\gamma}{|k|^{\tau}}, \\
      &|\langle k, \omega(\xi)\rangle \pm \Omega_{j}(\xi)|\geqslant \frac{\gamma|j|^{1+\delta}}{|k|^{\tau}},\\
      &|\langle k, \omega(\xi)\rangle\pm \Omega_{i}(\xi)\pm\Omega_{j}(\xi)|\geqslant \frac{\gamma(|i|+|j|)(|i|^{\delta}+|j|^{\delta})}{|k|^{\tau}},\\
      &|\langle k, \omega(\xi)\rangle+\Omega_{i}(\xi)-\Omega_{j}(\xi)|\geqslant \frac{\gamma(|i|-|j|)(|i|^{\delta}+|j|^{\delta})}{|k|^{\tau}},\\
      &|\langle k, \omega(\xi)\rangle+\Omega_{j}(\xi)-\Omega_{-j}(\xi)|\geqslant \frac{\gamma|j|^{\delta}}{|k|^{\tau}},
    \end{aligned}
    \begin{aligned}
      & k\neq 0,\\
      & \\
      & k\neq 0,\\
      & \\
      & k\neq 0,\\
      & \\
      & k\neq 0\ \text{and}\ |i|\neq |j|,\\
      & \\
      & k\neq 0.
    \end{aligned}
\end{equation*}
\begin{proposition}
$\meas (\mathcal {O}\setminus \mathcal {O}_{0})= O(\gamma)$.
\end{proposition}
\begin{proof}
We consider the following nonreasonant conditions:
\begin{equation*}
    \begin{split}
      &\langle k, \omega(\xi)\rangle \not\equiv 0, \ \ \ \forall k\in \mathbb{Z}^2,\\
      &\langle k, \omega(\xi)\rangle+ \Omega_{j}(\xi)\not\equiv 0, \ \ \ \forall k\in \mathbb{Z}^2,\\
      &\langle k, \omega(\xi)\rangle+ \Omega_{i}(\xi)+\Omega_{j}(\xi)\not\equiv 0, \ \ \ \forall k\in \mathbb{Z}^2,\\
      &\langle k, \omega(\xi)\rangle+ \Omega_{i}(\xi)-\Omega_{j}(\xi)\not\equiv 0, \ \ \ |k|+||i|-|j||\neq 0.
    \end{split}
\end{equation*}
As we have prove the first nonresonant condition in Proposition \ref{Nondegeneracy}, we only consider the remaining three conditions. One has to check that $\langle \alpha, k\rangle + \langle \beta,
l\rangle \neq 0$ or
$Ak+ B^{T}l \neq 0$ for $1\leqslant |l| \leqslant 2$.
Suppose $Ak+ B^{T}l=0$, for some $k\in \mathbb{Z}^2$ and $1\leqslant |l| \leqslant 2$.
We let $d$ be the sum of at most two nonzero components of $l$. Then
\begin{equation*}
    k_1(n_1-n_2)+n_1d=0,\ \ \
    k_2(n_2-n_1)+n_2d=0,
\end{equation*}
that is
\begin{equation*}
    k_1=\frac{n_1}{n_2-n_1}d,\ \ \
    k_2=-\frac{n_2}{n_2-n_1}d.
\end{equation*}
As $n_{1}$ is odd and $|n_{2}-n_{1}|=4$, we have the following integer solutions
\begin{equation*}
    d=0, \ \ \  k_{1}=k_{2}=0.
\end{equation*}
In this case, $k=0$, ``$l$'' have one ``$1$'' and one ``$-1$''. Because the perturbation admit compact form with respect to $n_1, n_2$, that is
\begin{equation*}
    k_1n_1+k_2n_1+il_i+jl_j=0,
\end{equation*}
then we get $i=j$. It contradicts with our assumption $i\neq j$. Thus, we prove all nonresonant conditions.

The desired measure estimate of $\meas (\mathcal {O}\setminus \mathcal {O}_{0})$ then follows from the same argument as that in Section 5.5.
\end{proof}

The next proposition is concerned with the Hamiltonian vector fields.
The perturbation term $P$ is real analytic in the space coordinates and Lipschitz in the parameters.
Moreover, near $\mathbb{T}_0^2$, for each $\xi\in \mathcal {O}$ its Hamiltonian vector field $X_P=(P_{I}, -P_{\theta}, -iP_{\bar{z}}, iP_z)^T$
defines a real analytic map
\begin{equation*}
    X_P:D(s,r)\times \mathcal {O} \longrightarrow \mathscr{P}^{a,q}.
\end{equation*}
For $s,r >0$, we introduce weighted norm for
$W=(X,Y,U,V)\in \mathscr{P}^{a,q}_{\mathbb{C}}$,
\begin{equation*}
    \|W\|_{r,q}=|X|+\frac{1}{r^2}|Y|+\frac{1}{r}\|U\|_{a,q}+\frac{1}{r}\|V\|_{a,q}.
\end{equation*}
Furthermore, for a map $W: D(s,r) \times \mathcal {O} \longrightarrow \mathscr{P}^{a,q}_{\mathbb{C}}$, for
example, the Hamiltonian vector field $X_P$, we define the norms
\begin{eqnarray*}
  \|W\|_{r,q,D(s,r) \times \mathcal {O}}^{\text{sup}}&=&\sup\limits_{D(s,r)\times \mathcal {O}} \|W\|_{r,q},\\
  \|W\|_{r,q,D(s,r) \times \mathcal {O}}^{\text{lip}}&=&\sup\limits_{\substack{\xi,\zeta\in \mathcal {O},\\ \xi\neq \zeta}}\sup\limits_{D(s,r)}\frac{\|\Delta _{\xi,\zeta}W\|_{r,q}}{|\xi-\zeta|}.
\end{eqnarray*}

\begin{proposition}\label{Regularity of perturbation}(Regularity of perturbation)
There exists a neighborhood $D(s,r)$ of $\mathbb{T}^2_0$ in $\mathscr{P}^{a,p}_{\mathbb{C}}$ such that $P$ is defined on $D(s,r) \times \mathcal {O}$,
and its Hamiltonian vector field defines a map
\begin{equation*}
    X_P:D(s,r)\times \mathcal {O} \longrightarrow \mathscr{P}^{a,p-1},
\end{equation*}
Moreover, $X_{P}(\cdot, \xi)$ is real analytic on $D(s,r)$ for each $\xi \in \mathcal {O}$, and
$X_{P}(\omega, \cdot)$ is uniformly Lipschitz on $\mathcal {O}$ for each $\omega \in D(s,r)$.
\end{proposition}
\begin{proof}
We first show that $P_{z}\in \ell^{a,p-1}$. From \eqref{perturbation1}, it is obvious that $\|P_{z}\|_{a,p-1}\leqslant c\varepsilon \|z\|_{a,p}$. The other components of $X_{P}$ is similar, and we get
$X_{P}\in \ell^{a,p-1}$. Because of the form of $\widetilde{P}$ in \eqref{perturbation} and $P$
in \eqref{perturbation1}, we know that $\|P_{z}(\omega, \cdot)\|_{a,p-1}^{\text{lip}}\leqslant c\varepsilon \|z\|_{a,p}$. $X_{P}$ is Lipschitz continuous on $\mathcal {O}$, for all $\omega \in
D(s,r)$.
\end{proof}

\begin{proposition}\label{The special form of the perturbation}(The special form of the perturbation) The perturbation $P$ in Hamiltonian \eqref{Hamiltonian3} belongs to $\mathcal {A}_{n_1,n_2}$.
\end{proposition}
\begin{proof}
Consider the Taylor-Fourier expansion of $P$: $P=\sum_{k,\alpha,\beta}P_{k,\alpha,\beta}(I)e^{\mi\langle k,\theta\rangle}z^{\alpha}\bar{z}^{\beta}$. It follows
from $K\in \mathcal {A}_{n_1,n_2}$ that $P\in \mathcal {A}_{n_1,n_2}$, i.e.
\begin{equation*}
    P_{k,\alpha,\beta}(I)=0,
\end{equation*}
whenever
\begin{equation*}
          k_1n_1+k_2n_2+\sum\limits_{n}(\alpha_n-\beta_n)n\neq 0
\end{equation*}
or
\begin{equation*}
          k_1+k_2+\sum\limits_{n}(\alpha_n-\beta_n)\neq 0.
\end{equation*}
This implies that $P$ contains no terms of the form $e^{\mi\langle k, \theta\rangle}z_{j}\bar{z}_{-j}$ with
$|j|> \frac{1}{2}\max \{|n_1|,|n_2|\}|k|$. Especially, $P$ contains no terms of the form $z_{j}\bar{z}_{-j}$.
Together with Lemma \ref{The special form}, there is no terms of the form $e^{\mi\langle k, \theta\rangle}z_n\bar{z}_n$ with $k\neq 0$ and $z_n\bar{z}_m$ with $n\neq m$ in $P$.
\end{proof}

\section{KAM step}
To begin with the KAM iteration, we first fixed $s, r>0 , p\geqslant\frac{3}{2}, d=2, \delta=1$ and restrict the Hamiltonian \eqref{Hamiltonian3} on the domain $D(s,r)$ and restrict
the parameter to the set $\mathcal {O}_{0}$. Initially, we set $\omega_{0}=\omega$, $\Omega_{0}=\Omega$, $P_{0}=P$, $r_{0}=r$, $s_{0}=0$ and $M_{0}=M$.
Consider the Hamiltonian :
\begin{equation*}
    N_{0}=\langle\omega_{0}(\xi),I\rangle+\sum\limits_{j\in \mathbb{Z}_1}\Omega_j^{0}(\xi)z_j\bar{z}_j, \ \ \ H_{0}=N_{0}+P_{0}.
\end{equation*}
Hence, $H_{0}$ is real analytic on $D(s_{0}, r_{0})$ and also depends on $\xi\in\mathcal {O}_{0}$ Whitney smoothly. It is obviously that there is constant $\varepsilon_{0}>0$ such that
\begin{equation*}
    \|X_{P}\|_{r,q,D(s_{0},r_{0}) \times \mathcal {O}_{0}}^{\sup}+ \frac{\gamma_{0}}{M_{0}}\|X_{P}\|_{r,q,D(s_{0},r_{0}) \times \mathcal {O}_{0}}^{\text{lip}}\leqslant \varepsilon_{0}.
\end{equation*}
where $\gamma_{0}=\varepsilon_{0}^{\frac{1}{3}}$.

We recall that
\begin{equation*}
    \begin{split}
    \mathcal {O}_{0}=\{&\xi\in \mathcal {O}: |\langle k, \omega_{0}(\xi)\rangle|\geqslant \frac{\gamma_{0}}{|k|^{\tau}}, k\neq 0;\\
    &|\langle k, \omega_{0}(\xi)\rangle\pm \Omega^{0}_{j}(\xi)|\geqslant \frac{\gamma_{0}|j|^{1+\delta}}{|k|^{\tau}}, k\neq 0;\\
    &|\langle k, \omega_{0}(\xi)\rangle\pm \Omega^{0}_{i}(\xi)\pm\Omega^{0}_{j}(\xi)|\geqslant \frac{\gamma_{0}(|i|+|j|)(|i|^{\delta}+|j|^{\delta})}{|k|^{\tau}}, k\neq 0;\\
    &|\langle k, \omega_{0}(\xi)\rangle+ \Omega^{0}_{i}(\xi)-\Omega^{0}_{j}(\xi)|\geqslant \frac{\gamma_{0}(|i|-|j|)(|i|^{\delta}+|j|^{\delta})}{|k|^{\tau}}, k\neq 0\, \text{and}\, |i|\neq |j|;\\
    &|\langle k, \omega_{0}(\xi)\rangle+ \Omega^{0}_{j}(\xi)-\Omega^{0}_{-j}(\xi)|\geqslant \frac{\gamma_{0}|j|^{\delta}}{|k|^{\tau}}, k\neq 0.\}
    \end{split}
\end{equation*}
and $P_{0}=\sum_{k,\alpha,\beta}P_{k,\alpha,\beta}^{0}(I)e^{\mi\langle k,\theta\rangle}z^{\alpha}\bar{z}^{\beta}\in\mathcal {A}_{n_1,n_2}$:
$k, \alpha, \beta$ have the following relations:
\begin{equation*}
          k_1n_1+k_2n_2+\sum\limits_{n}(\alpha_n-\beta_n)n=0\ \text{and}\ k_1+k_2+\sum\limits_{n}(\alpha_n-\beta_n)=0.
\end{equation*}
Suppose that after a $\nu$-th KAM step, we arrive at a Hamiltonian
\begin{equation*}
    H=H_{\nu}=N_{\nu}+P_{\nu}, \ \ \ N=N_{\nu}=\langle\omega_{\nu}(\xi),I\rangle+\sum\limits_{j\in \mathbb{Z}_1}\Omega_j^{\nu}(\xi)z_j\bar{z}_j,
\end{equation*}
which is real analytic in $(\theta, I, z\ \bar{z})\in D=D_{\nu}=D(s_{\nu}, r_{\nu})$ for some $s_{\nu}\leqslant s_{0}$, $r_{\nu}\leqslant r_{0}$ and
depends on $\xi\in \mathcal {O}_{\nu}\subset \mathcal {O}_{0}$ Whitney smoothly, where
\begin{equation*}
    \begin{split}
    \mathcal {O}_{\nu}=\{&\xi: |\langle k, \omega_{\nu}(\xi)\rangle|\geqslant \frac{\gamma}{|k|^{\tau}}, k\neq 0;\\
    &|\langle k, \omega_{\nu}(\xi)\rangle\pm \Omega_{j}^{\nu}(\xi)|\geqslant \frac{\gamma|j|^{1+\delta}}{|k|^{\tau}}, k\neq 0;\\
    &|\langle k, \omega_{\nu}(\xi)\rangle\pm \Omega_{i}^{\nu}(\xi)\pm\Omega_{j}^{\nu}(\xi)|\geqslant \frac{\gamma(|i|+|j|)(|i|^{\delta}+|j|^{\delta})}{|k|^{\tau}}, k\neq 0;\\
    &|\langle k, \omega_{\nu}(\xi)\rangle+ \Omega_{i}^{\nu}(\xi)-\Omega_{j}^{\nu}(\xi)|\geqslant \frac{\gamma(|i|-|j|)(|i|^{\delta}+|j|^{\delta})}{|k|^{\tau}}, k\neq 0\, \text{and}\, |i|\neq |j|;\\
    &|\langle k, \omega_{\nu}(\xi)\rangle+ \Omega_{j}^{\nu}(\xi)-\Omega_{-j}^{\nu}(\xi)|\geqslant \frac{\gamma|j|^{\delta}}{|k|^{\tau}}, k\neq 0.\}
    \end{split}
\end{equation*}
for some $\gamma_{\nu}\leqslant \gamma_{0}$. We also assume that
\begin{equation*}
    \|X_{P}\|_{r,q,D_{\nu}(s_{\nu},r_{\nu}) \times \mathcal {O}}^{\lambda_{\nu}}\leqslant \varepsilon_{\nu}.
\end{equation*}
for some $\lambda_{\nu}=\frac{\gamma_{\nu}}{M_{\nu}}$ and $0<\varepsilon_{\nu}\leqslant\varepsilon_{0}$.
$P=P_{\nu}=\sum_{k,\alpha,\beta}P_{k,\alpha,\beta}^{0}(I)e^{\mi\langle k,\theta\rangle}z^{\alpha}\bar{z}^{\beta}\in\mathcal {A}_{n_1,n_2}$:
$k, \alpha, \beta$ have the following relations:
\begin{equation*}
          k_1n_1+k_2n_2+\sum\limits_{n}(\alpha_n-\beta_n)n=0\ \text{and}\ k_1+k_2+\sum\limits_{n}(\alpha_n-\beta_n)=0.
\end{equation*}
We will construct a symplectic transformation $\Phi_{\nu}$, such that
\begin{equation*}
    H_{\nu+1}=H_{\nu}\circ\Phi_{\nu}=N_{\nu+1}+P_{\nu+1}
\end{equation*}
with another normal form $N_{\nu+1}$ and a smaller perturbation $P_{\nu+1}$ which defined on a smaller domain $D_{\nu+1}$.
We drop the index $\nu$ of $H_{\nu},N_{\nu},P_{\nu},\Phi_{\nu}$ and shorten the index $\nu+1$ as $+$.
Also, throughout the whole paper, we use letters c, C to denote suitable (possibly different) constants
that do not depend on the iteration steps.

\subsection{The Homological Equations}
Expand $P$ into the Fourier-Taylor series
\begin{equation*}
    P=\sum\limits_{k\in \mathbb{Z}^2, l\in\mathbb{N}^2, \alpha,\beta}P_{k,l,\alpha,\beta}(\xi)e^{\mi\langle k,\theta\rangle}I^{l}z^{\alpha}\bar{z}^{\beta}.
\end{equation*}
Let $R$ be the truncation of $P$ given by
\begin{equation}\label{truncation of P}
    R=\sum\limits_{2|l|+|\alpha+\beta|\leqslant 2}\sum\limits_{k\in \mathbb{Z}^2}P_{k,l,\alpha,\beta}(\xi)e^{\mi\langle k,\theta\rangle}I^{l}z^{\alpha}\bar{z}^{\beta}.
\end{equation}
The mean value of such a Hamiltonian is defined as
\begin{equation*}
    [R]=\sum\limits_{|l|+|\alpha|=1}P_{0,l,\alpha,\alpha}(\xi)I^{l}z^{\alpha}\bar{z}^{\alpha}
\end{equation*}
and is of the same form as $N$.

The coordinate transformation $\Phi$ is obtained as the time-1-map $X_{F}^{t}|_{t=1}$ of a Hamiltonian vector field $X_{F}$, where $F$
is the same form as $R$. Using the Taylor formula we can write
\begin{equation*}
    \begin{split}
    H\circ \Phi=N&\circ X_{F}^{1}+ R\circ X_{F}^{1}+(P-R)\circ X_{F}^{1}\\
               =N&+\{N,F\}+\int^1_0 (1-t)\{\{N,F\},F\}\circ X_{F}^{t}\,dt\\
                &+R+\int^1_0 \{R,F\}\circ X_{F}^{t}\,dt+(P-R)\circ X_{F}^{1}\\
               =N&+R+\{N,F\}+\int^1_0 \{R+(1-t)\{N,F\}, F\}\circ X_{F}^{t}\,dt+(P-R)\circ X_{F}^{1},
    \end{split}
\end{equation*}
where $\{\cdot,\cdot\}$ is the Possion bracket:
\begin{equation*}
    \{F,G\}=\sum\limits_{1\leqslant j\leqslant 2}\left(\frac{\partial F}{\partial \theta_j}\frac{\partial G}{\partial I_j}-\frac{\partial F}{\partial I_j}\frac{\partial G}{\partial \theta_j}\right)
    -\mi\sum\limits_{j\in \mathbb{Z}_{1}}\left(\frac{\partial F}{\partial z_j}\frac{\partial G}{\partial \bar{z}_j}-\frac{\partial F}{\partial \bar{z}_j}\frac{\partial G}{\partial z_j}\right).
\end{equation*}
In view of the previous equation, we define the new normal form by $N_{+} = N + \widehat{N}$, where $\widehat{N}$ satisfies
the so-called homological equation (the unknown are $F$ and $\widehat{N}$):
\begin{equation}\label{homological equation}
    \{F, N\}+ \widehat{N}=R.
\end{equation}
Once the homological equation is solved, we define the new perturbation term $P_{+}$ by
\begin{equation*}
    P_{+}=\int^1_0 \{R(t), F\}\circ X_{F}^{t}\,dt+(P-R)\circ X_{F}^{1},
\end{equation*}
where $R(t)=tR+(1-t)\widehat{N}$.
\begin{lemma}\label{solving homological equations}
The homological Equation \eqref{homological equation} has a solution $F$, $\widehat{N}$ which is unique with $[F]
= 0$, $[\widehat{N}] = \widehat{N}$, $F$ is regular on $D(s, r)\times\mathcal {O}$ in the above sense, and satisfies for $0 < \sigma < s$ the
estimates
\begin{equation*}
\begin{split}
    \|X_{F}\|_{r,p,D(s-\sigma,r)\times \mathcal{O}}^{\text{\emph{sup}}}&\leqslant \frac{C}{\gamma\sigma^{2\tau+3}}\|X_{R}\|_{r,q,D(s,r)\times \mathcal{O}}^{\text{\emph{sup}}},\\
    \|X_{F}\|_{r,p,D(s-\sigma,r)\times \mathcal{O}}^{\text{\emph{lip}}}&\leqslant \frac{C}{\gamma\sigma^{2\tau+3}}\left(\|X_{R}\|_{r,q,D(s,r)\times
    \mathcal{O}}^{\text{\emph{lip}}}+\frac{M}{\gamma}\|X_{R}\|_{r,q,D(s,r)\times \mathcal{O}}^{\text{\emph{sup}}}\right),
\end{split}
\end{equation*}
and
\begin{equation*}
\begin{split}
    \|X_{\widehat{N}}\|_{r,q,D(s-\sigma,r)\times \mathcal{O}}^{\text{\emph{sup}}}&\leqslant C\|X_{R}\|_{r,q,D(s,r)\times \mathcal{O}}^{\text{\emph{sup}}},\\
    \|X_{\widehat{N}}\|_{r,q,D(s-\sigma,r)\times \mathcal{O}}^{\text{\emph{lip}}}&\leqslant C\left(\|X_{R}\|_{r,q,D(s,r)\times
    \mathcal{O}}^{\text{\emph{lip}}}+\frac{M}{\gamma}\|X_{R}\|_{r,q,D(s,r)\times \mathcal{O}}^{\text{\emph{sup}}}\right).
\end{split}
\end{equation*}
\end{lemma}

\begin{proof}
Decompose  $R= R^{0}+R^{1}+R^{2}$, where $R^{j}$ comprises all terms in the expansion of $R$ with $|\alpha+\beta|=j$. Decompose similarly $F, N$ and $\widehat{N}$,
where necessarily $N^{1}=0$ and $\widehat{N}^{1}=0$ by normalization. The homological equation decomposes into
\begin{equation}\label{homological equation 1}
    \begin{split}
     &\{F^{0}, N\} + \widehat{N}^{0}=R^{0},\\
     &\{F^{1}, N\} =R^{1},\\
     &\{F^{2}, N\} + \widehat{N}^{2}=R^{2}.\\
    \end{split}
\end{equation}
We will see that with the chosen normalization and the Diophantine conditions these equations
determine $\widehat{N}^0, F^0, F^1$ and then $\widehat{N}^2, F^2 $ uniquely.

Due to independence of $z, \bar{z}$, the first equation amounts to the classical, finite-dimensional partial  differential equation
\begin{equation*}
    \partial_{\omega}F^{0}+\widehat{N}^{0}=R^{0}, \ \ \ \ \partial_{\omega}=\sum\limits_{1\leqslant i\leqslant 2} \omega_{i}\partial_{\theta_{i}}.
\end{equation*}
This leads to $\widehat{N}^{0}=[R^{0}]$ and $\partial_{\omega}F^{0}=R^{0}-[R^{0}]$ with $[F^{0}]=0$.
Their estimates are standard and of the same form --- indeed much better --- than the ones for $F^{1}, F^{2}$ and $\widehat{N}^{2}$ obtained below. For
later reference, we record that
\begin{equation*}
\begin{split}
    \|X_{F^{0}}\|_{r,p,D(s-2\sigma,r)\times \mathcal{O}}^{\text{sup}}&\leqslant \frac{C}{\gamma\sigma^{2\tau+3}}\|X_{R}\|_{r,q,D(s,r)\times \mathcal{O}}^{\text{sup}},\\
    \|X_{F^{0}}\|_{r,p,D(s-4\sigma,r)\times \mathcal{O}}^{\text{lip}}&\leqslant \frac{C}{\gamma\sigma^{2\tau+3}}\left(\|X_{R}\|_{r,q,D(s,r)\times
    \mathcal{O}}^{\text{lip}}+\frac{M}{\gamma}\|X_{R}\|_{r,q,D(s,r)\times \mathcal{O}}^{\text{sup}}\right).
\end{split}
\end{equation*}
Note that $X_{F^{0}}$ does not have any $z, \bar{z}$ component, so $\|X_{F^{0}}\|_{r,p}$ does not depend on $p$.

Consider the second equation in \eqref{homological equation 1}. Writing
\begin{equation*}
    R^{1}= R^{10} + R^{01} = \langle \mathcal {R}^{10}, z \rangle + \langle \mathcal {R}^{01}, \bar{z} \rangle
\end{equation*}
and similarly $F^{1}$ it decomposes into
\begin{equation*}
    \{F^{ij}, N\}= R^{ij}, \ \ \ i+j=1,
\end{equation*}
and it suffices to study each equation individually.

We have $\mathcal {R}^{10}= \left.R_{z}\right|_{z=\bar{z}=0}$ and thus
\begin{equation*}
    \frac{1}{r}\|\mathcal {R}^{10}\|_{a,q,D(s)}^{\text{sup}}\leqslant \|X_{R}\|_{r,q,D(s,r)}^{\text{sup}},
\end{equation*}
where $D(s)=\left\{\theta: |\text{Im}\ \theta|<s \right\}$. Writing $R^{10}=\langle \mathcal {R}^{10}, z \rangle= \sum_{j\in \mathbb{Z}_{1}} R_{j}(\theta, \xi)z_{j}$, and
similarly $F^{10}$, the equation $\{F^{10}, N\} =R^{10}$ further decomposes into
\begin{equation*}
    \partial_{\omega}F_{j} - \mi \Omega_{j}F_{j} = R_{j}, \ \ \ \ j\in \mathbb{Z}_{1}.
\end{equation*}
Expanding $R_{j}$ in a Fourier series,
$R_{j}=\sum_{k\in \mathbb{Z}^{n}\setminus \{0\}}\hat{R}_{jk}e^{i\langle k,\theta \rangle}$,
and similarly $F_{j}$. Then above homological equation can written in
\begin{equation*}
    \mi(\langle k,\omega \rangle - \Omega_{j})\hat{F}_{jk}=\hat{R}_{jk}, \ \ j\in \mathbb{Z}_{1}.
\end{equation*}

By the non-degeneracy condition \eqref{nonresonant condition} and Diophantine condition, we have uniformly
on $\mathcal {O}$
\begin{equation*}
    |\Omega_{j}(\xi)|\geqslant m |j|^d, \ \ \ \ j\in \mathbb{Z}_{1},
\end{equation*}
and
\begin{equation*}
    |\langle k, \omega(\xi)\rangle- \Omega_{j}(\xi)|\geqslant \frac{\gamma|j|^{1+\delta}}{|k|^{\tau}}, \ \ k\neq 0\ \text{and}\ j\in \mathbb{Z}_{1}.
\end{equation*}

Using the estimates in Lemma A.3, the unique solution $F_{j}$ satisfies the estimate
\begin{equation*}
    |F_{j}|^{\text{sup}}_{D(s-2\sigma)}\leqslant \frac{C}{\gamma \sigma^{\tau+1}|j|^{1+\delta}} |R_{j}|^{\text{sup}}_{D(s-\sigma)}, \ \ \ j\in \mathbb{Z}_{1}.
\end{equation*}
Since $p-q\leqslant \delta$, this and Lemma A.4 imply
\begin{equation*}
    \|\mathcal {F}^{10}\|^{\text{sup}}_{a,p,D(s-2\sigma)}\leqslant \frac{C}{\gamma \sigma^{\tau+1}} \|\mathcal {R}^{10}\|^{\text{sup}}_{a,q,D(s)}
    \leqslant \frac{C}{\gamma \sigma^{\tau+1}}r\|X_{R}\|^{\text{sup}}_{r,q,D(s,r)}.
\end{equation*}
The same estimate holds for $\mathcal {F}^{01}$. Multiplying $\mathcal {F}^{10}$ with $z$ and $\mathcal {F}^{01}$ with $\bar{z}$ and using $p>\frac{3}{2}$
this gives
\begin{equation*}
    \frac{1}{r^2}|F^{1}|^{\text{sup}}_{D(s-2\sigma,r)}\leqslant \frac{C}{\gamma \sigma^{\tau+1}}\|X_{R}\|^{\text{sup}}_{r,q,D(s,r)},
\end{equation*}
finally with Cauchy's estimate
\begin{equation*}
    \|X_{F^{1}}\|^{\text{sup}}_{r,a,p,D(s-3\sigma,r)}\leqslant \frac{C}{\gamma \sigma^{\tau+1}}\|X_{R}\|^{\text{sup}}_{r,q,D(s,r)}.
\end{equation*}

To obtain Lipschitz estimates, we study first the differences $\Delta F_{j}=F_{j}(\xi)-F_{j}(\zeta)$ for $\xi,\zeta \in \mathcal {O}$. we obtain
\begin{equation*}
    \partial_{\omega}\Delta F_{j} - \mi \Omega_{j}\Delta F_{j} = \Delta R_{j} +\partial_{\Delta\omega}F_{j} + \mi \Delta\Omega_{j}F_{j}, \ \ \ \ j\in \mathbb{Z}_{1}.
\end{equation*}
The right hand side is known, so $\Delta F_{j}$ uniquely solves the same kind of equation as $F_{j}$. So we obtain
\begin{equation*}
\begin{split}
    |\Delta F_{j}|^{\text{sup}}_{D(s-3\sigma)}&\leqslant \frac{C}{\gamma \sigma^{\tau+1}|j|^{1+\delta}} \left(|\Delta R_{j}|^{\text{sup}}_{D(s-\sigma)}
            + \frac{1}{\sigma}|F_{j}|^{\text{sup}}_{D(s-2\sigma)}(|\Delta \omega|+|\Delta \Omega_{j}|^{\text{sup}}_{D(s)})\right)\\
            &\leqslant \frac{C}{\gamma \sigma^{\tau+1}|j|^{1+\delta}} |\Delta R_{j}|^{\text{sup}}_{D(s-\sigma)}
            + \frac{C}{\gamma^2 \sigma^{2\tau+3}|j|^{2(1+\delta)}} |R_{j}|^{\text{sup}}_{D(s-\sigma)}(|\Delta \omega|+|\Delta \Omega_{j}|^{\text{sup}}_{D(s)}).
\end{split}
\end{equation*}
Then
\begin{equation*}
    \|\Delta \mathcal {F}^{10}\|^{\text{sup}}_{a,p,D(s-3\sigma)}
    \leqslant \frac{C}{\gamma \sigma^{\tau+1}} \|\Delta \mathcal {R}^{10}\|^{\text{sup}}_{a,q,D(s)}
    +\frac{C}{\gamma^2 \sigma^{2\tau+3}}\|\mathcal {R}^{10}\|^{\text{sup}}_{D(s)}(|\Delta \omega|+|\Delta \Omega|^{\text{sup}}_{-\delta,D(s)}).
\end{equation*}
Dividing by $|\xi-\zeta|\neq 0$ and taking the supremum over $\mathcal {O}$,
\begin{equation*}
\begin{split}
    \|\mathcal {F}^{10}\|^{\text{lip}}_{a,p,D(s-3\sigma)}
    &\leqslant \frac{C}{\gamma \sigma^{\tau+1}} \|\mathcal {R}^{10}\|^{\text{lip}}_{a,q,D(s)}
    +\frac{C}{\gamma^2 \sigma^{2\tau+3}}\|\mathcal {R}^{10}\|^{\text{sup}}_{D(s)}(|\omega|^{\text{lip}}_{\mathcal {O}}+ |\Omega|^{\text{lip}}_{-\delta,\mathcal {O}})\\
    &\leqslant \frac{C}{\gamma \sigma^{2\tau+3}} \left(\|\mathcal {R}^{10}\|^{\text{lip}}_{a,q,D(s)}
    +\frac{M}{\gamma}\|\mathcal {R}^{10}\|^{\text{sup}}_{D(s)}\right).
\end{split}
\end{equation*}
The same estimate applies to $\mathcal {F}^{01}$. So for the vector field of $F^{1}$, we finally get
\begin{equation*}
    \|X_{F^{1}}\|^{\text{lip}}_{r,p,D(s-4\sigma)}
    \leqslant \frac{C}{\gamma \sigma^{2\tau+3}} \left(\|X_{R}\|^{\text{lip}}_{r,q,D(s,r)}
    +\frac{M}{\gamma}\|X_{R}\|^{\text{sup}}_{r,q, D(s,r)}\right).
\end{equation*}
This concludes the discussion of $F^1$.

Now we consider the third equation in \eqref{homological equation 1}. Write $R^2 = R^{20} + R^{11} + R^{02}$ and similarly $F^2$ and
$N^2$. This equation decomposes into
\begin{equation}\label{homological equation 2}
    \{F^{ij}, N\} + \widehat{N}^{ij}=R^{ij},
\end{equation}
while $\widehat{N}^{ij}=0$ for $i\neq j$.

Consider the equation for $F^{11}$, which is slightly more complicated than the ones for $F^{20}$ and $F^{02}$.
Writing $R^{11}=\langle \mathcal {R}^{11}z, \bar{z}\rangle$, we have
$\left.\mathcal {R}^{11} = R_{z\bar{z}}\right|_{z=\bar{z}=0}$.
Thus, $R^{11}$ is the Jacobian of $R_{z}$ with respect to $\bar{z}$ at $\bar{z} = 0$. By Cauchy's inequality, we have
\begin{equation*}
    \|\mathcal {R}^{11}\|^{\text{sup}}_{q,p,D(s)}\leqslant \frac{1}{r}\|R_{z}\|^{\text{sup}}_{q,D(s,r)}\leqslant \|X_{R}\|^{\text{sup}}_{r,q,D(s,r)},
\end{equation*}
where $\|\cdot\|_{q,p}$ denotes the operator norm included by $\|\cdot\|_{a,p}$ and $\|\cdot\|_{a,p}$ in the source and target spaces,
respectively.

Note that P contains no terms of the form $z_{i}\bar{z}_{j}$ with $i\neq j$ and $e^{\mi\langle k, \theta\rangle}z_{j}\bar{z}_{j}$ with $k\neq 0$, write more explicitly
\begin{equation*}
    R^{11}=\sum\limits_{\substack{|i|\neq |j| \\ i,j \in \mathbb{Z}_1}}R_{ij}(\theta,\xi)z_{i}\bar{z}_{j}
    +\sum\limits_{j \in \mathbb{Z}_1}R_{jj}(\xi)z_{j}\bar{z}_{j}
    +\sum\limits_{\substack{j \in \mathbb{Z}_1 \\ |j|\leqslant \frac{1}{2}\max \{|n_1|,|n_2|\}|k|}}R_{j(-j)}(\theta,\xi)z_{j}\bar{z}_{-j},
\end{equation*}
and similarly $F^{11}$. The Eq.\,\eqref{homological equation 2} decomposes into
\begin{equation*}
    \partial_{\omega}F_{ij} - \mi (\Omega_{i}-\Omega_{j})F_{ij} = R_{ij}, \ \ \ \ i\neq j.
\end{equation*}
and $F_{jj} = 0$ for $j\neq 0$, $F_{j(-j)} = 0$ for $|j|>\frac{1}{2}\max \{|n_1|,|n_2|\}|k|$.

Again, by Diophantine condition , we have
\begin{equation*}
    |\langle k, \omega(\xi)\rangle+(\Omega_{i}(\xi)-\Omega_{j}(\xi))|\geqslant \frac{\gamma(|i|-|j|)(|i|^{\delta}+|j|^{\delta})}{|k|^{\tau}}, \ \ \ k\neq 0\ \text{and}\ |i|\neq |j|,
\end{equation*}
and
\begin{equation*}
    |\langle k, \omega(\xi)\rangle+(\Omega_{j}(\xi)-\Omega_{-j}(\xi))|\geqslant \frac{\gamma|j|^{\delta}}{|k|^{\tau}}, \ \ \ k\neq 0,
\end{equation*}
then we obtain
\begin{equation*}
    (|i|^{\delta}+|j|^{\delta})|F_{ij}|^{\text{sup}}_{D(s-2\sigma)}\leqslant \frac{C}{\gamma \sigma^{\tau+1}||i|-|j||} |R_{ij}|^{\text{sup}}_{D(s-\sigma,r)}, \ \ \ |i|\neq |j|,
\end{equation*}
and
\begin{equation*}
    |j|^{\delta}|F_{j(-j)}|^{\text{sup}}_{D(s-2\sigma)}\leqslant \frac{C}{\gamma\sigma^{\tau+1}}|R_{j(-j)}|^{\text{sup}}_{D(s-\sigma,r)}.
\end{equation*}
With Lemma A.4, this yield
\begin{equation*}
    \|\mathcal {F}^{11}\|^{\text{sup}}_{p,p,D(s-2\sigma)}, \|\mathcal {F}^{11}\|^{\text{sup}}_{q,q,D(s-2\sigma)}
    \leqslant \frac{C}{\gamma \sigma^{\tau+3}} \|\mathcal {R}^{11}\|^{\text{sup}}_{q,p,D(s)}
    \leqslant \frac{C}{\gamma \sigma^{\tau+3}}\|X_{R}\|^{\text{sup}}_{r,q,D(s,r)},
\end{equation*}
The same, and even better estimates hold for $F^{20}$ and $F^{02}$. Multiplying with $z$, $\bar{z}$ we then get
\begin{equation*}
    \frac{1}{r^2}|F^{2}|^{\text{sup}}_{D(s-2\sigma,r)}\leqslant \frac{C}{\gamma \sigma^{\tau+3}}\|X_{R}\|^{\text{sup}}_{r,q,D(s,r)},
\end{equation*}
finally with Cauchy's estimate
\begin{equation*}
    \|X_{F^{2}}\|^{\text{sup}}_{r,p,D(s-3\sigma,r)}\leqslant \frac{C}{\gamma \sigma^{\tau+3}}\|X_{R}\|^{\text{sup}}_{r,q,D(s,r)}.
\end{equation*}
The estimate for the Lipschitz semi-norm of $X_{F^2}$ is obtained by the same arguments as the one
for $X_{F^1}$ , and the result is analogous. We therefore omit it.

The estimates of $X_{\widehat{N}}$ follow from the observation that
\begin{equation*}
    \widehat{N}=\sum\limits_{|l|=1}P_{0l00}I^{l}+\sum\limits_{j\in \mathbb{Z}_{1}}P_{00jj}(\xi)z_{j}\bar{z}_{j}.
\end{equation*}
The final estimates of the lemma are obtained by replacing $\sigma$ by $\frac{\sigma}{4}$ throughout the proof.
\end{proof}
For $\lambda\geqslant 0$, we define
\begin{equation*}
    \|X\|^{\lambda}_{r}=\|X\|^{\text{sup}}_{r}+ \lambda \|X\|^{\text{lip}}_{r}.
\end{equation*}
The symbol ``$\lambda$'' in $\|X\|^{\lambda}_{r}$ will always be used in this role and never has the meaning of exponentiation.

\begin{lemma}\label{solving homological equations 1}The estimates of Lemma \ref{solving homological equations} imply that
\begin{equation*}
\begin{split}
    &\|X_{F}\|_{r,p,D(s-\sigma,r)\times \mathcal{O}}^{\lambda}\leqslant \frac{C}{\gamma\sigma^{2\tau+3}}\|X_{R}\|_{r,q,D(s,r)\times \mathcal{O}}^{\lambda},\\
    &\|X_{\widehat{N}}\|_{r,q,D(s-\sigma,r)\times \mathcal{O}}^{\lambda}\leqslant C\|X_{R}\|_{r,q,D(s,r)\times \mathcal{O}}^{\lambda},
\end{split}
\end{equation*}
for $0<\sigma<s$ and $0\leqslant\lambda\leqslant\frac{\gamma}{M}$ with another $C$ of the same form as in Lemma \ref{solving homological equations}.
\end{lemma}

The preceding lemma also gives us an estimate of $\|DX_{F}\|_{r,p,p,D(s-2\sigma,r)\times \mathcal{O}}^{\lambda}$ with the help of Cauchy's estimate.

\begin{lemma}Under the assumptions of Lemma \ref{solving homological equations},
\begin{equation*}
    \|DX_{F}\|_{r,p,p,D(s-2\sigma,r)\times \mathcal{O}}^{\lambda}, \|DX_{F}\|_{r,q,q,D(s-2\sigma,r)\times \mathcal{O}}^{\lambda}\leqslant \frac{C}{\gamma\sigma^{2\tau+3}}\|X_{R}\|_{r,q,D(s,r)\times
    \mathcal{O}}^{\lambda}.
\end{equation*}
\end{lemma}
\begin{proof} The proof can be found on page 160 in \cite{Kappeler}.
\end{proof}

We recall some approximation results in \cite{Poschel1}, which show that the second order approximation
of $P$ can be controlled by $P$, and that $P-R$ is small when we contract the domain (this
contraction is governed by the new parameter $\eta$).

\begin{lemma}Let $P$ satisfies Proposition \ref{Regularity of perturbation} and consider its Taylor approximation $R$ of the
form \eqref{truncation of P}. Then there exists $C > 0$  so that for all $\eta > 0$
\begin{eqnarray}\label{estimate of trunction}
  \|X_R\|^{\lambda}_{r,q,D(s,r)\times \mathcal{O}} &\leqslant & \|X_P\|^{\lambda}_{r,q,D(s,r)\times \mathcal{O}}, \label{estimate of trunction1}\\
  \|X_{P}-X_{R}\|^{\lambda}_{\eta r,q,D(s,4\eta r)\times \mathcal{O}} &\leqslant & C\eta\|X_P\|^{\lambda}_{r,q,D(s,r)\times \mathcal{O}}.\label{estimate of trunction2}
\end{eqnarray}
\end{lemma}

At the end,  we give some estimates for $X_{F}^{t}$. The formulas \eqref{estimate of flow 1} and \eqref{estimate of flow 2} will be used
to prove our coordinate transformation is well-defined. Inequalities \eqref{estimate of flow 3} and \eqref{estimate of flow 4} will be used to
check the convergence of the iteration.
\begin{lemma}\label{estimate of flow} If $\|X_P\|^{\lambda}_{r,q,D(s,r)\times \mathcal{O}}\leqslant \frac{\gamma\sigma^{2\tau+4}\eta^2}{C}$, we then have
\begin{equation}\label{estimate of flow 1}
    X_{F}^{t}: D(s-2\sigma, \frac{r}{2}) \longrightarrow D(s-\sigma, r), \ \ -1\leqslant t\leqslant 1.
\end{equation}
Similarly,
\begin{equation}\label{estimate of flow 2}
    X_{F}^{t}: D(s-3\sigma, \frac{r}{4}) \longrightarrow D(s-2\sigma, \frac{r}{2}), \ \ -1\leqslant t\leqslant 1.
\end{equation}
Moreover,
\begin{equation}\label{estimate of flow 3}
    \|X_{F}^{t}- Id\|_{r,p,D(s-2\sigma,\frac{r}{2})\times \mathcal{O}}^{\lambda}< C \|X_{F}\|_{r,p,D(s-\sigma,r)\times \mathcal{O}}^{\lambda},
\end{equation}
\begin{equation}\label{estimate of flow 4}
    \|DX_{F}^{t}- Id\|_{r,q,q,D(s-3\sigma,\frac{r}{4})\times \mathcal{O}}^{\lambda}< C \|DX_{F}\|_{r,q,q,D(s-\sigma,r)\times \mathcal{O}}^{\lambda},
\end{equation}
for $0\leqslant\lambda\leqslant\frac{\gamma}{M}$. The latter estimate also holds in the $\|\cdot\|_{r,p,p}$ - norm.
\end{lemma}
We can use Lemma \ref{solving homological equations 1} and Lemma A.5 to prove this lemma.

\subsection{The New Hamiltonian}

The map $\Phi=X_{F}^{1}$ defined above transforms $H$ into $H \circ \Phi=N_{+}+P_{+}$ on $D(s-\sigma, \frac{r}{2})$,
where $N_{+}=N+\widehat{N}$ and
\begin{equation*}
    P_{+}=\int^1_0 \{R(t),F\}\circ X_{F}^{t}\,dt+(P-R)\circ X_{F}^{1},
\end{equation*}
where $R(t)=tR+(1-t)\widehat{N}$. Hence
\begin{equation*}
    X_{P_{+}}=\int^1_0 (X_{F}^{t})^{*}[X_{R(t)}, X_F]\,dt+(X_{F}^{1})^{*}(X_P-X_R).
\end{equation*}
From the paper \cite{Poschel1},  we have known the following result:
\begin{equation}\label{estimate of vector field}
    \|(X_F^t)^*Y\|^{\lambda}_{\eta r,q,D(s-4\sigma,\eta r)} \leqslant C\|Y\|^{\lambda}_{\eta r,q,D(s-2\sigma,4\eta r)}, \ \ 0\leqslant t\leqslant 1.
\end{equation}
We already have estimated $\|X_{P}-X_{R}\|^{\lambda}_{\eta r,q}$ in \eqref{estimate of trunction2}, so it remains to consider the commutator
$\|[X_{R(t)}, X_F]\|_{r,q}$.

First, we have
\begin{equation*}
 \begin{split}
     \|X_{R(t)}\|^{\lambda}_{r,q,D(s-\sigma,r)}&\leqslant \|X_{\widehat{N}}\|^{\lambda}_{r,q,D(s-\sigma,r)}+ \|X_{R}\|^{\lambda}_{r,q,D(s-\sigma,r)}\\
     &\leqslant C\|X_{P}\|^{\lambda}_{r,q,D(s,r)}.
 \end{split}
\end{equation*}
Moreover, we have the pointwise estimate
\begin{equation*}
 \begin{split}
     \|[X_{R(t)}, X_F]\|_{r,q}&\leqslant \|DX_{R(t)}\cdot X_F\|_{r,q} + \|X_{R(t)}\cdot DX_F\|_{r,q}\\
     &\leqslant \|DX_{R(t)}\|_{r,q,p}\|X_{F}\|_{r,p}+\|DX_{F}\|_{r,q,q}\|X_{R(t)}\|_{r,q}.
 \end{split}
\end{equation*}
By the product rule for Lipschitz-norms and Cauchy's estimate, we thus obtain
\begin{equation*}
 \begin{split}
     \|[X_{R(t)}, X_F]\|^{\lambda}_{r,q, D(s-2\sigma,\frac{r}{2})}
     \leqslant &\|DX_{R(t)}\|^{\lambda}_{r,q,p,D(s-2\sigma,\frac{r}{2})}\|X_{F}\|^{\lambda}_{r,p,D(s-2\sigma,\frac{r}{2})}\\
     &+\|DX_{F}\|^{\lambda}_{r,q,q,D(s-2\sigma,\frac{r}{2})}\|X_{R(t)}\|^{\lambda}_{r,q,D(s-2\sigma,\frac{r}{2})}\\
     \leqslant &\frac{C}{\gamma \sigma^{2\tau + 1}}\left(\|X_{P}\|^{\lambda}_{r,q,D(s,r)}\right)^2,
 \end{split}
\end{equation*}
for $0\leqslant\lambda\leqslant\frac{\gamma}{M}$. Hence, also
\begin{equation*}
 \begin{split}
     \|[X_{R(t)}, X_F]\|^{\lambda}_{\eta r,q, D(s-2\sigma,\frac{r}{2})}
     \leqslant & \frac{1}{\eta^2}\|[X_{R(t)}, X_F]\|^{\lambda}_{r,q, D(s-2\sigma,\frac{r}{2})}\\
     \leqslant &\frac{C}{\gamma \sigma^{2\tau + 3}\eta^2}\left(\|X_{P}\|^{\lambda}_{r,q,D(s,r)}\right)^2.
 \end{split}
\end{equation*}

Together with the estimate on$\|X_{P}-X_{R}\|^{\lambda}_{\eta r,q}$ in \eqref{estimate of trunction2} and with that in \eqref{estimate of vector field}, we finally arrive at the
estimate
\begin{equation*}
   \|X_{P_{+}}\|^{\lambda}_{\eta r,q, D(s-2\sigma,\frac{r}{2})}\leqslant C\eta\|X_P\|^{\lambda}_{r,q,D(s,r)} + \frac{C}{\gamma \sigma^{2\tau +
   3}\eta^2}\left(\|X_{P}\|^{\lambda}_{r,q,D(s,r)}\right)^2,
\end{equation*}
for $0\leqslant\lambda\leqslant\frac{\gamma}{M}$. This is the bound for the new perturbation.

Now turn to the new frequencies $\omega_{+}(\xi)=\omega(\xi)+\widehat{\omega}(\xi)$ and $\Omega_{+}(\xi)=\Omega(\xi)+\widehat{\Omega}(\xi)$. For $\widehat{N}$, we
have the estimate
\begin{equation*}
    \|X_{\widehat{N}}\|_{r,q,D(s-\sigma,r)\times \mathcal{O}}^{\lambda}\leqslant C\|X_{R}\|_{r,q,D(s,r)\times \mathcal{O}}^{\lambda},
\end{equation*}
for $0\leqslant\lambda\leqslant\frac{\gamma}{M}$. The weighted norm implies that we have $|\widehat{\omega}(\xi)|\leqslant \|X_{\widehat{N}}\|^{\text{sup}}_{r,q}$ and
$\|\widehat{\Omega}(\xi)z\|_{a,q}\leqslant r\|X_{\widehat{N}}\|^{\text{sup}}_{r,q}$ on $D(s,r)$ and consequently $|\widehat{\Omega}(\xi)|_{q-p}\leqslant \|X_{\widehat{N}}\|^{\text{sup}}_{r,q}$. The
same holds for the Lipschitz semi-norms. Since $p-q\leqslant \delta$, we obtain
\begin{equation}\label{new frequencies norm}
    |\widehat{\omega}|^{\lambda}_{\mathcal {O}} +  |\widehat{\Omega}|^{\lambda}_{-\delta, \mathcal {O}}\leqslant C\|X_{P}\|_{r,q,D(s,r)\times \mathcal{O}}^{\lambda},
\end{equation}
where $\Omega=(\Omega_{j})_{j\in \mathbb{Z}_{1}}$ and $\widehat{\Omega}=(\widehat{\Omega}_{j})_{j\in \mathbb{Z}_{1}}$.

In order to control the assumptions of the KAM step for the iteration, we notice that the last
estimate also implies
\begin{equation}\label{estimate of new frequence}
    |\langle l,\widehat{\Omega}(\xi)\rangle|
    \leqslant \langle l\rangle_{\delta}|\widehat{\Omega}|_{-\delta}
    \leqslant \langle l\rangle_{d-1}\|X_{P}\|_{r,q,D(s,r)\times \mathcal{O}}^{\lambda}.
\end{equation}

\begin{lemma}$P_{+}\in \mathcal {A}_{n_1,n_2}$.
\end{lemma}
\begin{proof}Note that
\begin{equation*}
\begin{split}
    P_{+}=&P-R+\{P, F\}+\frac{1}{2!}\{\{N, F\}, F\}+\frac{1}{2!}\{\{P, F\}, F\}\\
      &+ \cdots + \frac{1}{n!}\{\cdots\{N, \underbrace{F\}\cdots, F}\limits_{n}\} + \frac{1}{n!}\{\cdots\{P, \underbrace{F\}\cdots, F}\limits_{n}\}+ \cdots.
\end{split}
\end{equation*}
Since $P\in \mathcal {A}_{n_1,n_2}$, then $F$, so do $P-R$, $\{N, F\}$ and $\{P, F\}$. The lemma follows
from Lemma \ref{compact form} and Lemma \ref{gauge invariant property}.
\end{proof}

\subsection{Iteration Lemma}
For given $\varepsilon_{0}$, $m_{0} = m$,
$M_{0}=M$, $r_{0}=r$, $s_{0}=s<1$. Moreover, we define
sequences as follows:
\begin{equation*}
    \begin{split}
    \sigma_{0}=&\ \frac{s}{8}, \ \ \ \ \sigma_{\nu+1}=\frac{\sigma_{\nu}}{2}, \ \ \ \ s_{\nu+1}=s_{\nu}-2\sigma_{\nu},\\
    \eta^{3}_{\nu}=&\ \frac{\varepsilon_{\nu}}{\gamma_{\nu}\sigma_{\nu}^{2\tau+3}}, \ \ \ \ r_{\nu+1}=\eta_{\nu}r_{\nu}, \ \ \ \ D_{\nu}=D(s_{\nu}, r_{\nu}),\\
    \gamma_{0}=&\ \varepsilon_{0}^{\frac{1}{3}}, \ \ \ \ \gamma_{\nu}=\varepsilon_{\nu}^{\frac{1}{3}},
    \ \ \ \ M_{\nu}=\ M_{0}(2-2^{\nu}), \ \ \ \ \lambda_{\nu}=\frac{\gamma_{\nu}}{M_{\nu}},\\
    m_{\nu}=&\ \frac{m_0}{2}(1+2^{-\nu}), \ \ \ \ \varepsilon_{\nu+1}=C(\gamma_{\nu}\sigma_{\nu}^{2\tau+3})^{-\frac{1}{3}}\varepsilon_{\nu}^{\frac{4}{3}},\\
    \end{split}
\end{equation*}
\begin{equation*}
    \begin{split}
    \mathcal {O}_{\nu}=\{&\xi\in \mathcal {O}_{\nu-1}: |\langle k, \omega_{\nu}(\xi)\rangle|\geqslant \frac{\gamma_{\nu}}{|k|^{\tau}}, k\neq 0;\\
    &|\langle k, \omega_{\nu}(\xi)\rangle\pm \Omega^{\nu}_{j}(\xi)|\geqslant \frac{\gamma_{\nu}|j|^{1+\delta}}{|k|^{\tau}}, k\neq 0;\\
    &|\langle k, \omega_{\nu}(\xi)\rangle\pm \Omega^{\nu}_{i}(\xi)\pm\Omega^{\nu}_{j}(\xi)|\geqslant \frac{\gamma_{\nu}(|i|+|j|)(|i|^{\delta}+|j|^{\delta})}{|k|^{\tau}}, k\neq 0;\\
    &|\langle k, \omega_{\nu}(\xi)\rangle+ \Omega^{\nu}_{i}(\xi)-\Omega^{\nu}_{j}(\xi)|\geqslant \frac{\gamma_{\nu}(|i|-|j|)(|i|^{\delta}+|j|^{\delta})}{|k|^{\tau}}, k\neq 0\, \text{and}\, |i|\neq
    |j|;\\
    &|\langle k, \omega_{\nu}(\xi)\rangle+ \Omega^{\nu}_{j}(\xi)-\Omega^{\nu}_{-j}(\xi)|\geqslant \frac{\gamma_{\nu}|j|^{\delta}}{|k|^{\tau}}, k\neq 0\}.
    \end{split}
\end{equation*}
The proceeding analysis can be summarized as follows:
\begin{lemma}\label{iterative lemma}
Let
\begin{equation*}
    \varepsilon_{0}\leqslant \frac{\gamma_{0}\sigma_{0}^{2\tau+6}}{C^3}, \ \   \gamma_{0}\leqslant \frac{m_0}{2}.
\end{equation*}
Suppose, $H_{\nu}=N_{\nu}+P_{\nu}$ is given on $D(s_{\nu}, r_{\nu})\times \mathcal {O}_{\nu}$ which is real analytic in $(\theta, I, z, \bar{z})\in D(s_{\nu}, r_{\nu})$ and
Whitney smooth in $\xi\in \mathcal {O}_{\nu}$, where
\begin{equation*}
    N_{\nu}=\langle\omega_{\nu}(\xi),I\rangle+\sum\limits_{j\in \mathbb{Z}_1}\Omega^{\nu}_j(\xi)z_j\bar{z}_j.
\end{equation*}
Its coefficients satisfy
\begin{equation*}
\begin{split}
    |\omega_{\nu}|^{\emph{\text{lip}}}_{\mathcal {O}_{\nu}}+ |\Omega^{\nu}|^{\emph{\text{lip}}}_{-\delta,\mathcal {O}_{\nu}}&\leqslant M_{\nu},\\
    |\omega_{\nu}-\omega_{\nu-1}|^{\lambda_{\nu}}_{\mathcal {O}_{\nu}}&\leqslant \varepsilon_{\nu-1},\\
    |\Omega^{\nu}-\Omega^{\nu-1}|^{\lambda_{\nu}}_{-\delta,\mathcal {O}_{\nu}}&\leqslant \varepsilon_{\nu-1},\\
\end{split}
\end{equation*}
and
\begin{equation}\label{estimate of spectral}
    |\langle l,\Omega^{\nu}(\xi)\rangle|\geqslant m_{\nu} \langle l\rangle_{d-1}, \ \ \forall 1\leqslant |l|\leqslant 2.
\end{equation}
$P_{\nu}\in \mathcal {A}_{n_{1},n_{2}}$, and
\begin{equation*}
    \|X_{P_{\nu}}\|^{\lambda_{\nu}}_{r_{\nu},q,D(s_{\nu},r_{\nu})\times \mathcal {O}_{\nu}}\leqslant \varepsilon_{\nu}.
\end{equation*}
Then there exists a family of symplectic coordinate transformation
\begin{equation*}
    \Phi_{\nu+1}: D_{\nu+1}\times \mathcal {O}_{\nu} \longrightarrow D_{\nu}
\end{equation*}
and a closed subset
\begin{equation*}
    \mathcal{O}_{\nu+1}=\mathcal {O}_{\nu}\setminus \bigcup\limits_{|k|>0,l}\mathfrak{R}^{\nu+1}_{k,l}(\gamma_{\nu+1}),
\end{equation*}
where
\begin{equation*}
    \begin{split}
    \mathfrak{R}^{\nu+1}_{k,l}(\gamma_{\nu+1})=\{&\xi\in \mathcal {O}_{\nu}: |\langle k, \omega_{\nu+1}(\xi)\rangle|< \frac{\gamma_{\nu+1}}{|k|^{\tau}}; \\
    &|\langle k, \omega_{\nu+1}(\xi)\rangle\pm \Omega^{\nu+1}_{j}(\xi)|< \frac{\gamma_{\nu+1}|j|^{1+\delta}}{|k|^{\tau}};\\
    &|\langle k, \omega_{\nu+1}(\xi)\rangle\pm \Omega^{\nu+1}_{i}(\xi)\pm\Omega^{\nu+1}_{j}(\xi)|< \frac{\gamma_{\nu+1}(|i|+|j|)(|i|^{\delta}+|j|^{\delta})}{|k|^{\tau}};\\
    &|\langle k, \omega_{\nu+1}(\xi)\rangle+ \Omega^{\nu+1}_{i}(\xi)-\Omega^{\nu+1}_{j}(\xi)|< \frac{\gamma_{\nu+1}(|i|-|j|)(|i|^{\delta}+|j|^{\delta})}{|k|^{\tau}},
    |i|\neq |j|;\\
    &|\langle k, \omega_{\nu+1}(\xi)\rangle+ \Omega^{\nu+1}_{j}(\xi)-\Omega^{\nu+1}_{-j}(\xi)|< \frac{\gamma_{\nu+1}|j|^{\delta}}{|k|^{\tau}}.\}
    \end{split}
\end{equation*}
such that for $H_{\nu+1}=H_{\nu}\circ \Phi_{\nu+1}=N_{\nu+1}+P_{\nu+1}$ the same assumptions as above are satisfied
with ``$\nu+1$'' in place of ``$\nu$''.
\end{lemma}
\begin{proof} By induction one verifies that $\varepsilon_{\nu}\leqslant \frac{\gamma_{\nu}\sigma_{\nu}^{2\tau+6}}{C^3}$. With the definition of $\eta_{\nu}$, namely
$\eta^{3}_{\nu}=\frac{\varepsilon_{\nu}}{\gamma_{\nu}\sigma_{\nu}^{2\tau+3}}$, this
implies $\varepsilon_{\nu}\leqslant\frac{\gamma_{\nu}\sigma_{\nu}^{2\tau+4}\eta_{\nu}^2}{C}$. By the KAM step there exists a transformation $\Phi_{\nu+1}: D_{\nu+1}\times \mathcal {O}_{\nu}
\longrightarrow D_{\nu}$ taking $H_{\nu}$
into $H_{\nu+1}=N_{\nu+1}+P_{\nu+1}$.
The new perturbation $P_{\nu+1}$ then satisfies the estimate
\begin{equation*}
  \begin{split}
   \|X_{P_{\nu+1}}\|^{\lambda_{\nu+1}}_{r_{\nu+1},q,D(s_{\nu+1},r_{\nu+1})\times \mathcal {O}_{\nu+1}}
   &\leqslant C\eta_{\nu} \varepsilon_{\nu}+ \frac{C}{\gamma_{\nu} \sigma_{\nu}^{2\tau + 3}\eta_{\nu}^2}\varepsilon_{\nu}^2\\
   &=C(\gamma_{\nu}\sigma_{\nu}^{2\tau+3})^{-\frac{1}{3}}\varepsilon_{\nu}^{\frac{4}{3}}=\varepsilon_{\nu+1}.
  \end{split}
\end{equation*}
In view of \eqref{new frequencies norm} the Lipschitz semi-norm of the new frequencies is bounded by
\begin{equation*}
    M_{\nu}+ C\|X_{P_{\nu}}\|^{\lambda_{\nu}}_{r_{\nu},q,D(s_{\nu},r_{\nu})\times \mathcal {O}_{\nu}}\leqslant M_{\nu} + \frac{C\varepsilon_{\nu}}{\gamma_{\nu}}M_{\nu}
    \leqslant M_{\nu}(1+2^{-\nu-2})\leqslant M_{\nu+1}
\end{equation*}
as required.
By the estimate \eqref{estimate of new frequence}, that is,
\begin{equation*}
    |\langle l,\Omega^{\nu+1}-\Omega^{\nu}(\xi)\rangle|\leqslant \varepsilon_{\nu}\langle l\rangle_{d-1}
    \leqslant \frac{\gamma_{\nu}\sigma_{\nu}^{2\tau+6}}{C^3}\langle l\rangle_{d-1}
    \leqslant \frac{m_{0}}{2^{\nu+2}}\langle l\rangle_{d-1},
\end{equation*}
then $|\langle l,\Omega^{\nu+1}(\xi)\rangle|\geqslant m_{\nu+1}\langle l\rangle_{d-1}$ for $0<|l|\leqslant 2$.
\end{proof}

\subsection{Convergence}
Let
\begin{equation*}
    \Psi^{\nu}=\Phi_{1}\circ\Phi_{2}\circ\cdots\circ\Phi_{\nu}: D_{\nu}\times \mathcal {O}_{\nu-1} \longrightarrow D_{0}.
\end{equation*}
Inductively, we have that
\begin{equation*}
    H_{\nu}=H\circ \Psi^{\nu}=\langle\omega_{\nu}(\xi),I\rangle+\sum\limits_{j\in \mathbb{Z}_1}\Omega^{\nu}_j(\xi)z_j\bar{z}_j+P_{\nu}(\theta, I ,z,\bar{z},\xi).
\end{equation*}
Note that $ \varepsilon_{\nu} \rightarrow 0$ as $\nu\rightarrow +\infty$, we can make the KAM estimates go on well at each step.
Let $\mathcal {O}_{\ast}=\bigcap^{\infty}_{\nu=0}\mathcal {O}_{\nu}$. As in \cite{Poschel1}, thanks to Lemma \ref{estimate of flow}, it concludes that $N_{\nu}$, $\omega_{\nu}$,
$\Omega^{\nu}$, $\Psi^{\nu}$ and $D\psi^{\nu}$ converge uniformly on $D(\frac{s}{2}, 0) \times \mathcal {O}_{\ast}$ with
\begin{equation*}
     N_{\infty}=\langle\omega_{\infty}(\xi),I\rangle+\sum\limits_{j\in \mathbb{Z}_1}\Omega^{\infty}_j(\xi)z_j\bar{z}_j.
\end{equation*}
Let $X_{H}^{t}$ be the flow of $X_{H}$. Since $H_{\nu}=H\circ \Psi^{\nu}$, we have
\begin{equation}\label{transform rule}
    X_{H}^{t}\circ \Psi^{\nu}= \Psi^{\nu}\circ X_{H_{\nu}}^{t}.
\end{equation}
The uniform convergence of $\Psi^{\nu}$, $D\Psi_{\nu}$, $X_{H_{\nu}}^{t}$ implies that the limits can be taken on the both sides
of \eqref{transform rule}. Hence, on $D(\frac{s}{2}, 0) \times \mathcal {O}_{\ast}$, we get
\begin{equation}
    X_{H}^{t}\circ \Psi^{\infty}= \Psi^{\infty}\circ X_{H_{\infty}}^{t}.
\end{equation}
and
\begin{equation}\label{transform rule1}
    \Psi^{\infty}: D(\frac{s}{2}, 0) \times \mathcal {O}_{\ast} \longrightarrow D(s, r) \times \mathcal {O},
\end{equation}
it follows from \eqref{transform rule1} that we get an invariant finite dimensional tori $\Psi^{\infty}(\mathbb{T}^{2}\times \{\xi\})$ for the original
perturbed Hamiltonian system at $\xi\in \mathcal {O}$. We remark that the frequencies $\omega_{\ast}(\xi)=\omega_{\infty}(\xi)$ associated
with $\Psi^{\infty}(\mathbb{T}^{2}\times \{\xi\})$ are slightly deformed from the unperturbed ones $\omega(\xi)$. The normal behaviors
of the invariant tori $\Psi^{\infty}(\mathbb{T}^{2}\times \{\xi\})$ are governed by their respective normal frequencies $\Omega_{n}^{\infty}$.	

\subsection{Measure Estimate}
For each $|k|>0$, we denote
\begin{equation*}
    \begin{split}
    \mathcal {R}^{\nu+1}_{k}=\{&\xi\in \mathcal {O}_{\nu}: |\langle k, \omega_{\nu+1}(\xi)\rangle|< \frac{\gamma_{\nu+1}}{|k|^{\tau}}\},\\
    \mathcal {R}^{\nu+1}_{kj}=\{&\xi\in \mathcal {O}_{\nu}: |\langle k, \omega_{\nu+1}(\xi)\rangle\pm \Omega^{\nu+1}_{j}(\xi)|< \frac{\gamma_{\nu+1}|j|^{1+\delta}}{|k|^{\tau}}\},\\
    \mathcal {R}^{\nu+1}_{kij}=\{&\xi\in \mathcal {O}_{\nu}: |\langle k, \omega_{\nu+1}(\xi)\rangle\pm \Omega^{\nu+1}_{i}(\xi)\pm\Omega^{\nu+1}_{j}(\xi)|<
    \frac{\gamma_{\nu+1}(|i|+|j|)(|i|^{\delta}+|j|^{\delta})}{|k|^{\tau}}\}, \\
    \overline{\mathcal {R}}^{\nu+1}_{kij}=\{&\xi\in \mathcal {O}_{\nu}: |\langle k, \omega_{\nu+1}(\xi)\rangle+ \Omega^{\nu+1}_{i}(\xi)-\Omega^{\nu+1}_{j}(\xi)|<
    \frac{\gamma_{\nu+1}(|i|-|j|)(|i|^{\delta}+|j|^{\delta})}{|k|^{\tau}}, \ \ |i|\neq |j|\},\\
    \overline{\mathcal {R}}^{\nu+1}_{kj(-j)}=\{&\xi\in \mathcal {O}_{\nu}: |\langle k, \omega_{\nu+1}(\xi)\rangle+ \Omega^{\nu+1}_{j}(\xi)-\Omega^{\nu+1}_{-j}(\xi)|<
    \frac{\gamma_{\nu+1}|j|^{\delta}}{|k|^{\tau}}\},
    \end{split}
\end{equation*}
then
\begin{equation*}
    \mathfrak{R}^{\nu+1}_{k,l}(\gamma_{\nu+1})=\mathcal {R}^{\nu+1}_{k}\cup\bigcup_{i,j}\left(\mathcal {R}^{\nu+1}_{kj}\cup\mathcal {R}^{\nu+1}_{kij}\cup\overline{\mathcal
    {R}}^{\nu+1}_{kij}\right)\cup\bigcup_{|j|\leqslant \frac{1}{2}\max\{|n_{1}, |n_{2}|\}|k|}\overline{\mathcal {R}}^{\nu+1}_{kj(-j)}.
\end{equation*}
At each step, we have to exclude the following
resonant set:
\begin{equation*}
    \mathfrak{R}^{\nu+1}=\bigcup\limits_{|k|>0,l}\mathfrak{R}^{\nu+1}_{k,l}(\gamma_{\nu+1}),
\end{equation*}
then
\begin{equation*}
    \mathcal {O}\setminus \mathcal {O}_{\ast}= \bigcup\limits_{\nu\geqslant 0}\mathfrak{R}^{\nu+1}.
\end{equation*}
Note that
\begin{equation*}
    \mathfrak{R}^{\nu+1}_{k,l}\setminus\bigcup_{|j|\leqslant \frac{1}{2}\max\{|n_{1}, |n_{2}|\}|k|}\overline{\mathcal {R}}^{\nu+1}_{kj(-j)}\subset \widetilde{\mathfrak{R}}^{\nu+1}_{k,l},
\end{equation*}
where
\begin{equation*}
    \widetilde{\mathfrak{R}}^{\nu+1}_{k,l}(\gamma_{\nu+1})=\{\xi\in \mathcal {O}_{\nu}:
    |\langle k, \omega_{\nu+1}(\xi)\rangle + \langle l, \Omega^{\nu+1}(\xi)\rangle| < \frac{\gamma_{\nu+1}\langle l\rangle_{\delta}}{|k|^{\tau}}\}.
\end{equation*}
Now we will prove that the measure of set $\widetilde{\mathfrak{R}}^{\nu+1}_{k,l}$ is small, so does $\mathfrak{R}^{\nu+1}_{k,l}$.
\begin{lemma}\label{measure estimate 1}
If $\widetilde{\mathfrak{R}}^{\nu}_{k,l}(\gamma_{\nu})\neq \emptyset$, then
\begin{equation*}
    \langle l\rangle_{d-1}\leq c|k|,
\end{equation*}
where $c=4(1+|\omega|^{\text{sup}}_{\mathcal {O}})/m$ is independent of $\nu$.
\end{lemma}
\begin{proof}
If there exists $\xi \in \widetilde{\mathfrak{R}}^{\nu}_{k,l}(\gamma_{\nu})$, then \eqref{estimate of spectral} implies that for $k\neq 0$,
\begin{equation*}
\begin{split}
    |\langle k, \omega_{\nu}(\xi)\rangle|&\geqslant |\langle l, \Omega^{\nu}(\xi)\rangle|-\gamma_{\nu}\frac{\langle l\rangle_{\delta}}{|k|^{\tau}},\\
    &\geqslant m_{\nu}\langle l\rangle_{\delta}- \gamma_{\nu}\langle l\rangle_{\delta},\\
    &\geqslant \frac{m}{4}\langle l\rangle_{d-1}
\end{split}
\end{equation*}
since $\langle l\rangle_{\delta}\leqslant \langle l\rangle_{d-1}$ for $\delta\leqslant d-1$ and $\gamma_{\nu}\leqslant \frac{m_{\nu}}{2}$, $m_{\nu}\geqslant \frac{m}{2}$
by construction. Hence,
\begin{equation*}
    \frac{m}{4}\langle l\rangle_{d-1}\leqslant |k||\omega_{\nu}(\xi)|\leqslant |k|(1+|\omega|^{\text{sup}}_{\mathcal {O}}).
\end{equation*}
\end{proof}

\begin{lemma}\label{measure estimate 2}
For fixed $\nu+1$, $k$, $l$,
\begin{equation*}
    \meas \widetilde{\mathfrak{R}}^{\nu+1}_{k,l}(\gamma_{\nu+1})< C\rho_{\nu}\frac{\gamma_{\nu+1}}{|k|^{\tau+1}},
\end{equation*}
where $\rho_{\nu}$ is the diameter of $\mathcal {O}_{\nu}$.
\end{lemma}
\begin{proof}Denote
\begin{equation*}
    f(\xi)=\langle k, \omega_{\nu+1}(\xi)\rangle + \langle l, \Omega^{\nu+1}(\xi)\rangle,
\end{equation*}
let vector $\nu$ satisfy $\langle k, \nu \rangle=|k|$. It follows that
\begin{equation*}
    \frac{d f(\xi+t\nu)}{dt}\geqslant C|k|>0,
\end{equation*}
where $C$ is some positive constant. Then by using Lemma A.6, it is easy to prove that estimate. So
we omit it here.
\end{proof}
\begin{lemma} \label{measure estimate 4} For fixed $\nu+1$, $k$, $l$,
\begin{equation*}
    \meas \left(\bigcup_{|j|\leqslant \frac{1}{2}\max\{|n_{1}, |n_{2}|\}|k|}\overline{\mathcal {R}}^{\nu+1}_{kj(-j)}\right)\leqslant C\rho_{\nu}\frac{\gamma_{\nu+1}}{|k|^{\tau}}
\end{equation*}
\end{lemma}
\begin{proof}Like Lemma \ref{measure estimate 2}, we have that
\begin{equation*}
    \left|\frac{\partial (\langle k, \omega_{\nu+1}(\xi)\rangle+ \Omega^{\nu+1}_{j}(\xi)-\Omega^{\nu+1}_{-j}(\xi))}{\partial \xi}\right|\geqslant C|k|>0.
\end{equation*}
By Lemma A.6, we know
\begin{equation*}
    \meas \overline{\mathcal {R}}^{\nu+1}_{kj(-j)}\leqslant C\rho_{\nu}|j|^{\delta}\frac{\gamma_{\nu+1}}{|k|^{\tau+1}}.
\end{equation*}
Since $|j|\leqslant \frac{1}{2}\max\{|n_{1}, |n_{2}|\}|k|$ and $\delta=1$, we obtain
\begin{equation*}
    \meas \left(\bigcup_{|j|\leqslant \frac{1}{2}\max\{|n_{1}, |n_{2}|\}|k|}\overline{\mathcal {R}}^{\nu+1}_{kj(-j)}\right)\leqslant C\rho_{\nu}\frac{\gamma_{\nu+1}}{|k|^{\tau}}.
\end{equation*}
\end{proof}

\begin{lemma}\label{measure estimate 3}
For fixed $\nu+1\geqslant 0$,
\begin{equation*}
    \meas \left(\bigcup\limits_{|k|>0, l}\widetilde{\mathfrak{R}}^{\nu+1}_{k,l}(\gamma_{\nu+1})\right)\leqslant C\rho_{\nu}\gamma_{\nu+1},
\end{equation*}
\begin{equation*}
    \meas \left(\bigcup\limits_{|k|>0, l}\bigcup_{|j|\leqslant \frac{1}{2}\max\{|n_{1}, |n_{2}|\}|k|}\overline{\mathcal {R}}^{\nu+1}_{kj(-j)}\right)\leqslant C\rho_{\nu}\gamma_{\nu+1},
\end{equation*}
where $C$ is a constant.
\end{lemma}
\begin{proof}For a fixed $k$, it suffices to consider $l$ with $\langle l \rangle_{d-1}\leqslant c|k|$ according to Lemma \ref{measure estimate 1}. Taking
into account that $|l|_{d-1}\leqslant 2\langle l \rangle_{d-1}$, we get
\begin{equation*}
    \card \{l: |l|\leqslant 2, \langle l \rangle_{d-1}\leqslant c|k|\}\leqslant c|k|^{s}, \ \ s=\frac{2}{d-1}.
\end{equation*}
Hence, by Lemma \ref{measure estimate 2} and \ref{measure estimate 4},
\begin{equation*}
    \meas \left(\bigcup\limits_{l}\widetilde{\mathfrak{R}}^{\nu+1}_{l}(\gamma_{\nu+1})\right)\leqslant C\rho_{\nu}\frac{\gamma_{\nu+1}}{|k|^{\tau-s}}.
\end{equation*}
If we choose $\tau\geqslant s+3$, then
\begin{equation*}
    \meas \left(\bigcup\limits_{|k|>0, l}\widetilde{\mathfrak{R}}^{\nu+1}_{k,l}(\gamma_{\nu+1})\right)\leqslant C\rho_{\nu}\gamma_{\nu+1}.
\end{equation*}
Similarly, we can prove the second measure estimate. So, Lemma \ref{measure estimate 3} follows.
\end{proof}
By Lemma \ref{measure estimate 3}, we can obtain the following result about the finite dimension Lebesgue measure
of $(\mathcal {O}_{\nu}\setminus \mathcal {O}_{\nu+1})$, i.e.,
\begin{equation*}
\begin{split}
    \meas (\mathcal {O}_{\nu}\setminus \mathcal {O}_{\nu+1})&= \meas \left(\bigcup\limits_{|k|>0,l}\mathfrak{R}^{\nu+1}_{k,l}(\gamma_{\nu+1})\right)\\
    &\leqslant \meas \left(\bigcup\limits_{|k|>0,l}\widetilde{\mathfrak{R}}^{\nu+1}_{k,l}(\gamma_{\nu+1})\right)
    +\meas \left(\bigcup\limits_{|k|>0, l}\bigcup_{|j|\leqslant \frac{1}{2}\max\{|n_{1}, |n_{2}|\}|k|}\overline{\mathcal {R}}^{\nu+1}_{kj(-j)}\right)\\
    &= O(\gamma_{\nu+1})\longrightarrow 0,
\end{split}
\end{equation*}
as $\nu \longrightarrow \infty $. It follows that the measure of all excluded parameters can be as small as we wish.

Finally we get a Cantor-like parameter set $\mathcal {O}_{\ast}=\bigcap^{\infty}_{\nu=0}\mathcal {O}_{\nu}$ of positive Lebesgue measure.

\section{Appendix}
\vspace{2mm}
In this section, we give some  technical lemmas.

\vspace{4mm}
\emph{Lemma A.1} Generalized Cauchy inequalities
\begin{equation*}
    \|F_{\theta}\|_{D(s-\sigma,r)}\leqslant \frac{c}{\sigma}\|F\|_{D(s,r)}, \ \ \ \|F_{I}\|_{D(s,\frac{r}{2})}\leqslant \frac{c}{r^2}\|F\|_{D(s,r)},
\end{equation*}
\begin{proof} The proof can be found in \cite{Kappeler}, \cite{Poschel2}.
\end{proof}

\vspace{4mm}
\emph{Lemma A.2}  For \,$\nu>0,0<\delta<1$, we have
\begin{equation*}
     \sum\limits_{k\in \mathbb{Z}^n}|k|^{\nu}e^{-2|k|\delta}\leqslant \left(\frac{\nu}{e}\right)^{\nu}\frac{1}{\delta^{\nu+n}}(1+e)^n.
\end{equation*}
\begin{proof} The inequality can be found on page 22 in \cite{Bogoljubov}.
\end{proof}

\vspace{4mm}
\emph{Lemma A.3} Let $u_{j}, j\geqslant 1$, be complex functions on $\mathbb{T}^{n}$ that are real analytic on $D(s)=\{|Im\,x|<s\}$. Then
\begin{equation*}
    \left(\sum\limits_{j\geqslant 1}\sup_{x\in D(s-\sigma)}|u_{j}(x)|^2\right)^{\frac{1}{2}}\leqslant
    \frac{4^{n}}{\sigma^n}\sup_{x\in D(s)}\left(\sum\limits_{j\geqslant 1}|u_{j}(x)|^2\right)^{\frac{1}{2}},
\end{equation*}
for  $0<\sigma\leqslant s\leqslant 1$.
\begin{proof} The proof can be found on page 262--263 in \cite{Kappeler}.
\end{proof}

\emph{Lemma A.4} Let $A=(A_{ij})_{i,j\neq 0}$ be a bounded operator on $\ell^2$ which depends on $x\in \mathbb{T}^n$ such that
all coefficients are analytic on $D(s)=\{|Im\,x|<s\}$. Suppose $B=(B_{ij})_{i,j\neq 0}$ is another operator on $\ell^2$ depending on $x$ whose coefficients
satisfy
\begin{equation*}
    \sup\limits_{x\in D(s)}|B_{ij}(x)|\leqslant \frac{1}{||i|-|j||}\sup\limits_{x\in D(s)}|A_{ij}(x)|,\ \ |i|\neq |j|,
\end{equation*}
and $B_{jj}=0, B_{-jj}$ for $j\neq 0$. Then $B$ is a bounded operator on $\ell^2$ for every $x\in D(s)$,
\begin{equation*}
    \sup\limits_{x\in D(s-\sigma)}\|B(x)\|\leqslant \frac{4^{n+1}}{\sigma^n}\sup\limits_{x\in D(s)}\|A(x)\|,
\end{equation*}
for $0<\sigma\leqslant s\leqslant 1$.

\begin{proof} The proof can be found on page 262--263 in \cite{Kappeler}.
\end{proof}

Let $V$ be an open domain in a real Banach space $E$ with norm $\|\cdot\|$, $\Pi$ a subset of another real
Banach space, and $X: V\times \Pi \rightarrow E$ a parameter dependent vector field on $V$, which is $C^1$
on $V$ and Lipschitz on $B$. Let $\phi^{t}$ be its flow. Suppose there is a subdomain $U\subset V$ such that
$\phi^{t}: V\times \Pi \rightarrow E$ for $-1\leqslant t\leqslant 1$.

\vspace{4mm}
\emph{Lemma A.5} Under the preceding assumptions,
\begin{equation*}
    \begin{split}
    &\|\phi^{t}- id \|_{U} \leqslant \|X\|_{V},\\
    &\|\phi^{t}- id \|_{U}^{\text{lip}}\leqslant exp(\|DX\|_{V})\|X\|_{V}^{\text{lip}},
    \end{split}
\end{equation*}
for $-1\leqslant t\leqslant 1$, where all norms are understood to be taken also over $\Pi$.
\begin{proof} The proof can be found in \cite{Poschel1}.
\end{proof}

\emph{Lemma A.6} Suppose that $g(u)$ is a $C^{N}$ function on the closure $\bar{I}$ , where $I\subset \mathbb{R}^{1}$ is an interval.
Let $I_h=\{u: |g(u)|\leqslant h\}, h>0$. If for some constant $d>0$, $|g^{N}(u)|\geqslant d $ for $\forall u \in I$, then
$|I_h|\leqslant c h^{\frac{1}{N}}$.where $|I_h|$ denotes the Lebesgue measure of $I_h$ and the constant $c=2(2+3+\ldots+N+d^{-1})$.

\vspace{2mm}
\emph{Remark}: In fact, if $N=1$, then $c=2d^{-1}$; if $N=2$, then $c=2(2+d^{-1})$; if $N \geqslant 3$, then $c=2(2+3+\ldots+N+d^{-1})$.
\begin{proof} The proof can be found in \cite{You}.
\end{proof}

\end{document}